\documentclass[11pt,a4paper]{article}
\usepackage[bottom=2.2cm,top=2.0cm,left=2.0cm,right=2.0cm]{geometry}

\usepackage{graphicx} 
\usepackage{amsthm}
\usepackage{amsmath}
\usepackage{amssymb}
\usepackage{amsfonts}
\usepackage{mathrsfs}
\usepackage{mathtools}
\mathtoolsset{showonlyrefs}
\usepackage{stmaryrd}
\usepackage{bm}
\usepackage{verbatim}
\usepackage[dvipsnames]{xcolor}
\usepackage{cancel}
\usepackage{enumitem}
\numberwithin{equation}{section}
\usepackage{tikz}
\usepackage{todonotes}
\usepackage{float}
\usepackage{multirow}

\usepackage{xifthen}

\usepackage{tikz}
\usepackage{footmisc}
\usetikzlibrary{positioning}
\usepackage{pgfplots}
\usetikzlibrary{external}
\usetikzlibrary{shapes}
\usetikzlibrary{positioning,arrows.meta}
\pgfplotsset{compat=newest} 
\usepgfplotslibrary{units} 
\usepackage{caption}
\usepackage{subcaption}
\usepackage{tabularx}

\usepackage{tcolorbox}
\definecolor{mioColore}{HTML}{E6E6FF}
\newtcolorbox{mioBox}{
  colback=mioColore,   
  colframe=blue!60!black,   
  arc=8pt,                  
  boxrule=1pt,              
  left=6pt, right=6pt, top=6pt, bottom=6pt
}
\definecolor{mioCommento}{HTML}{6666CC}
\definecolor{myblue}{HTML}{2437A6}%

\newcommand{\pa}[1]{{\color{black} #1}}

\theoremstyle{definition}
\newtheorem{definition}{Definition}
\newtheorem{assumption}{Assumption}
\theoremstyle{theorem}
\newtheorem{theorem}{Theorem}
\newtheorem{lemma}{Lemma}
\newtheorem{corollary}{Corollary}
\newtheorem{rmk}{Remark}
\newtheorem{proposition}{Proposition}

\newcommand{\bbf}[1]{\mathbf{#1}}
\newcommand{\scr}[1]{\mathscr{#1}}
\newcommand{\mcal}[1]{\mathcal{#1}}
\newcommand{\averagel}{\{\!\!\{}
\newcommand{\averager}{\}\!\!\}}

\newcommand{\jumpl}{[\![}
\newcommand{\jumpr}{]\!]}

\newcommand{\vh}[2]{\ifthenelse{\isempty{#2}}{\bbf{#1}_h}{\bbf{#1}_h^{\mcal{#2}}}}
\newcommand{\ph}[2]{\ifthenelse{\isempty{#2}}{#1_h}{#1_h^{\mcal{#2}}}}
\newcommand{\evh}[2]{\ifthenelse{\isempty{#2}}{\bbf{#1}_h}{\bbf{#1}_h^{#2}}}
\newcommand{\eph}[2]{\ifthenelse{\isempty{#2}}{#1_h}{#1_h^{#2}}}
\newcommand{\vph}[2]{\ifthenelse{\isempty{#2}}{\bbf{#1}_{\ph{p}{}}}{\bbf{#1}_{\ph{p}{#2}}}}

\newcommand{\ah}[3]{\ifthenelse{\isempty{#3}}{a_h(#1,#2)}{a_h^#3(#1,#2)}}
\newcommand{\bh}[3]{\ifthenelse{\isempty{#3}}{b_h(#1,#2)}{b_h^#3(#1,#2)}}
\newcommand{\sh}[2]{s_h(#1, #2)}

\newcommand{\normx}[2]{\ifthenelse{\isempty{#2}}{\|#1\|_{\bbf{X}_h}}{\|#1\|_{\bbf{X}_h^#2}}}
\newcommand{\normll}[2]{\ifthenelse{\isempty{#2}}{\|#1\|_{L^2(\Omega)}}{\|#1\|_{L^2(\Omega_#2)}}}
\newcommand{\tnormx}[1]{|\!|\!|#1|\!|\!|_{\bbf{X}_h}}
\newcommand{\normp}[2]{\ifthenelse{\isempty{#2}}{\|#1\|_{M_h}}{\|#1\|_{M_h^#2}}}
\newcommand{\tnormp}[1]{|\!|\!|#1|\!|\!|_{M_h}}

\newcommand{\tnormdiv}[1]{{|\!|\!|#1|\!|\!|_{div}}}
\newcommand{\tnorme}[1]{|\!|\!|#1|\!|\!|_{E}}
\newcommand{\tnormxq}[1]{|\!|\!|#1|\!|\!|^2_{\bbf{X}_h}}
\newcommand{\tnormdivq}[1]{{|\!|\!|#1|\!|\!|^2_{div}}}
\newcommand{\tnormeq}[1]{|\!|\!|#1|\!|\!|^2_{E}}
\newcommand{\tnormpq}[1]{|\!|\!|#1|\!|\!|^2_{M_h}}

\title{\pa{Polytopal Discontinuous Galerkin Discretizations of \\ Coupled Non-Newtonian Stokes--Darcy Systems}$^*$}

\author{
  Paola F. Antonietti$^{a,1}$ \and
  Michele Botti$^{a,2}$ \and
  Nicola Parolini$^{a,3}$ \and
  Valentina Pederzoli$^{a,4}$ \and
  Marco Verani$^{a,5}$
}
\date{}

\begin{document}

\maketitle
\begingroup
\renewcommand\thefootnote{}
\footnotetext{$^*$\textbf{Funding}: This research has been funded by the European Union (ERC, NEMESIS, project number 101115663). Views and opinions expressed are, however, those of the author(s) only and do not necessarily reflect those of the European Union or the European Research Council Executive Agency. The present research is part of the activities of Dipartimento di Eccellenza 2023-2027. VP, MB, NP, MV and PFA are members of INdAM-GNCS.}
\footnotetext{$^a$ MOX--Dipartimento di Matematica, Politecnico di Milano, Piazza Leonardo da Vinci 32, Milan, 20133, Italy}
\footnotetext{$^1$paola.antonietti@polimi.it}
\footnotetext{$^2$michele.botti@polimi.it}
\footnotetext{$^3$nicola.parolini@polimi.it}
\footnotetext{$^4$valentina.pederzoli@polimi.it}
\footnotetext{$^5$marco.verani@polimi.it}
\endgroup
\begin{abstract}
We propose and analyze a polytopal discontinuous Galerkin method for the numerical approximation of a coupled non-Newtonian Stokes--Darcy system modeling the interaction between a non-Newtonian free-flow fluid and a non-Newtonian flow through a porous medium. Due to its geometric flexibility and arbitrary-order accuracy, the proposed discretization scheme is well-suited to configurations with complex geometries. We provide a complete \emph{a-priori} analysis that considers shear-dependent and velocity-dependent non-Newtonian viscosity models for the free-flow and porous media regions, respectively. The well-posedness, stability, and error bounds of the method are established in the framework of generalized inf-sup theory. Error estimates are confirmed by numerical results.
\end{abstract}

\section{Introduction}

Coupled non-Newtonian Stokes--Darcy systems model fluid flow in interconnected free-flow and porous-medium regions. 
More precisely, the motion of a viscous fluid in an open region unobstructed by any porous solid matrix is modeled by the nonlinear Stokes equations in the incompressible creeping laminar regime whereas the flow within the porous medium is governed by a nonlinear Darcy-type constitutive relation that links the seepage velocity to the pressure gradient via the permeability tensor.  At the interface, the two flow regimes are coupled through physically consistent transmission conditions, consisting of the Beavers--Joseph--Saffman condition \cite{Beavers_Joseph_1967,Saffman,JagerMikelic_BJS}, along with continuity of the normal flux and equilibrium of normal stresses across the interface. Coupled non-Newtonian Stokes--Darcy models appear in a wide range of applied science and industrial applications. Among the most relevant natural settings are subsurface flows in karst aquifers, where water circulating through underground conduits and fractures continuously interacts with the surrounding porous matrix, a system whose accurate modeling is essential for sustainable groundwater resource management and for assessing vulnerability to contamination \cite{CGHW,Discacciati_ApplicationThesis,Cesmelioglu_Applications}. In addition, biological flows, such as plasma filtration through capillary walls and fluid transport in biological tissues, provide further examples of such coupled flow phenomena. A wide range of industrial applications also falls within this modeling framework. For example, cross-flow and dead-end filtration processes are widely employed in the pharmaceutical, chemical, food processing, and aeronautical industries \cite{HWNW1, HWNW2,NumericalAnalysisCoupStokesDarcyFlowsIndustrialFitrations,badia_quaini_quarteroni_2008,Quaini_thesis}.  In these applications, the fluid typically exhibits non-Newtonian behavior, meaning that its constitutive response is governed by nonlinear laws. Unlike Newtonian fluids, the viscosity of these fluids depends on the strain rate, resulting in a nonlinear relationship between shear stress and the rate of deformation. They are typically classified into two categories: shear-thinning fluids, such as blood and melted polymers, whose viscosity decreases with increasing shear rate, and shear-thickening fluids, such as cornstarch-water mixtures, whose viscosity increases with increasing shear rate. We refer, for example, to \cite{book:nonnewtonian, DUNN1995689} for a comprehensive overview of non-Newtonian fluids. 

Given the applicative importance of coupled non-Newtonian Stokes and Darcy systems, many numerical methods have been proposed and analyzed in the literature for their approximate solution. In \cite{BonaldiBrennerDroniouMasson}, the authors consider a two-phase Darcy flow in a fractured porous medium solved with a gradient discretization method. In \cite{DistaccatiMiglioQuarteroni}, the coupling between the Navier--Stokes and a Darcy equation is discussed along with suitable interface conditions, and then it is solved using a Finite Element (FEM) approximation, while in \cite{ErvinJenkinsSunFEM}, Ervin, Jenkins, and Sun developed and analyzed a Finite Element Method for a coupled nonlinear Stokes--Darcy system with generalized constitutive laws in both the free-flow and porous-medium regions, establishing rigorous error estimates for the proposed approximation. In \cite{ErvinJenkinsSunMORTARFEM}, a mortar Finite Element formulation for the coupled problem is presented; the key idea is to reformulate the coupling as an interface matching problem, where an interface pressure (or Lagrange multiplier) is introduced on the Stokes--Darcy interface and approximated using a mortar finite element space. In \cite{RiviereYotov}, a locally conservative coupling has been studied using a mixed FEM for the Darcy flow and a Discontinuous Galerkin (DG) method for the Stokes flow, while in \cite{KanschatRiviere2010}, the authors propose a new numerical method employing divergence-conforming velocity spaces. DG approximations have also been analyzed in \cite{RiviereSD,GiraultRiviere} for Newtonian constitutive laws, both adopting a primal formulation of the Darcy equation and mixed-order elements. In \cite{CongreveHouston2014}, Congreve and Houston propose a two-grid $hp$-version DG method for quasi-Newtonian fluid flows. A non-Newtonian Stokes–Darcy–Forchheimer model discretized via a DG method on triangular grids has been presented in \cite{SDFDG}, employing $P^{l+1}/P^l$ discontinuous elements. The corresponding error analysis is carried out within a Lagrange multiplier framework, using $P^{l+1}$ discontinuous elements for the multiplier. In \cite{DistaccatiQuarteroni} an overview of some results on the coupling of Navier--Stokes and Darcy's equations is presented. Polytopal Discontinuous Galerkin (PolyDG) methods on general polygonal and polyhedral meshes have attracted considerable attention in recent years due to their flexibility in handling complex geometries and support mesh agglomeration techniques. Such methods have been successfully developed and analyzed for (linear) Darcy flows in fractured porous media, see, e.g., \cite{AntoniettiFacciolaRussoVerani2019, Antonietti2020} and for (linear) Stokes flow problems, with also applications to fluid-structure interaction, see, e.g., \cite{YeZhang2020, AntoniettiMascottoVeraniZonca2022}. In \cite{LipnikovVassilevYotov}, Lipnikov, Vassilev and Yotov present a PolyDG formulation for coupled, linear, Stokes–Darcy flows, while Li, Gao, Zhang and Chen, present a PolyDg method for the solution of a linear Stokes--Darcy--Darcy (bulk-fracture) model in \cite{LiGaoZhangChen_SDD}. For earlier developments of PolyDG methods for elliptic problems, we refer, e.g.,  to \cite{AntoniettiGianiHouston_hpCompositeDG, BassiBottiColomboDiPietroTesini_flexibilityagglomeration, CangianiDongHoustonGeorgoulisHouston, CangianiDongHoustonGeorgoulisHouston}; see also the monograph \cite{CangianiDongGeorgoulisHouston2017}, the review paper \cite{AntoniettiCangianiCollisDongGeorgoulisGianiHouston2016} and the references therein. \\

The aim of this work is to propose and analyze a PolyDG method for discretizing a coupled non-Newtonian Stokes--Darcy system, supplemented by the (physically consistent) Beaver--Joseph--Saffman interface condition, along with continuity of the normal flux and equilibrium of normal stresses across the interface. In the context of coupled Stokes-Darcy problems, it is convenient that the numerical method treats both the free flow and the porous flow equations within a unified and consistent framework, avoiding the need for different discretization strategies in the two subdomains. In particular, it is desirable that the method yields the same order of convergence in both regions, so that the accuracy of the overall coupled solution is not limited by a discrepancy in the approximation quality between the two sides of the interface \cite{BurmanHansbo_stabilizedSD}. In the context of the model under consideration, a PolyDG approach offers a particularly appealing numerical framework. Indeed, the discontinuous nature of the approximation space allows for local tuning of approximation parameters, such as the polynomial degree $p$ and the mesh size $h$, which proves especially beneficial when dealing with non-conforming (possibly agglomerated) meshes. Furthermore, the use of polytopal meshes provides remarkable geometric flexibility, making PolyDG particularly advantageous in coupled porous-fluid problems, where the physical domain typically comprises (highly) heterogeneous subregions with irregular (or curved) interfaces. 
Moreover, the weak enforcement of inter-element continuity via numerical fluxes guarantees a natural embedding of the interface conditions directly in the formulation and built-in stability in advection-dominated and heterogeneous regimes \cite{SobolevEmbedding, RiviereDG,AntoniettiBonaldiMazzieri}.
 In this paper,  we first analyze the well-posedness of the continuous model, and then propose and analyze an equal-order mixed PolyDG method for its numerical approximation. On the one hand, the equal-order choice for the discontinuous discretization spaces is appealing for the practical implementation of the method, but on the other hand, it introduces two additional challenges from the analysis viewpoint: establishing the well-posedness of the discrete problem and deriving a suitable inf-sup condition for the coupled system that accounts for the pressure stabilization terms required to ensure stability when equal-order polynomials are used. To address the first challenge, we extend the results of \cite{MappedCoercivity} to our discrete framework. The inf-sup condition is then established by extending a generalized inf-sup condition with pressure stabilization for the Stokes equation \cite{AntoniettiMascottoVeraniZonca2022} to the Darcy setting, via an approach based on space inclusions and norm ordering. A further consequence of the mixed formulation framework is that the penalty terms associated with the Darcy vector and scalar variables mirror their Stokes counterparts: the penalty for the vector variable scales as $O(h^{-1})$, while that for the scalar variable scales as $O(h)
$, cf. \cite{AntoniettiMascottoVeraniZonca2022}. Finally, we consider an Incomplete Interior Penalty (IIP) formulation and a simple fixed-point iteration scheme to solve the nonlinear system.  For the proposed discretization, we prove its well-posedness,  stability bounds, and we show that the approximation error measured in a suitable energy norm (defined as the sum of the DG norm of the Stokes' and Darcy's variables) of the error scales as $O(h^l)$, being $l$ the polynomial order chosen for all the discrete variables.\\

The remainder of the manuscript is organized as follows. Section~\ref{sec: model and weak formulation} presents the model problem, introduces the assumptions on the physical parameters, derives the weak formulation of the continuous problem, and establishes its well-posedness. Section~\ref{sec: discrete problem} introduces the PolyDG discrete framework and presents the derivation of the numerical scheme. Sections \ref{sec: well-posedness} and \ref{sec: error estimates} are devoted to the analysis of well-posedness and convergence analysis, respectively. Finally, Section~\ref{sec: numerical results} presents numerical convergence results for both Newtonian and non-Newtonian test cases, and in Section~\ref{sec: Conclusion} we draw the conclusions of our work, and discuss some possible future developments.

\section{Model problem and its well-posedness}\label{sec: model and weak formulation}
In this section, we introduce the mathematical model governing the coupled shear-thinning, non-Newtonian fluid-flow system. The free-flow region, where a non-Newtonian viscous fluid moves unobstructed in an open channel, is described by the Stokes equations for incompressible creeping laminar flow, which balance viscous stresses with pressure forces while neglecting inertial effects. The flow within the adjacent porous medium is instead governed by Darcy's law, which relates the seepage velocity linearly to the pressure gradient through the permeability tensor of the medium. The two flow regimes occupy distinct but adjacent subdomains, separated by an interface, and are coupled by conditions that enforce continuity of physical quantities across the interface. Specifically, the conservation of mass across the interface is ensured by the continuity of the normal component of the velocity, while the balance of normal forces is imposed through the continuity of the normal stress. The tangential behavior at the interface is described by the Beavers--Joseph--Saffman condition, which relates the tangential component of the viscous stress in the free flow region to the tangential velocity of the free fluid at the interface, neglecting the contribution of the seepage velocity in the porous medium. On the Stokes' boundary, we consider homogeneous Dirichlet boundary conditions (no-slip condition), assuming that the velocity of the fluid is equal to the velocity of the wall. Instead, on the Darcy's boundary, we assume the normal component of the seepage velocity to be zero, meaning that there is no fluid flow across that boundary \cite{SDFDG}.\\

We assume that the computational domain $\Omega \subseteq \mathbb{R}^d, \ d = 2,3$, is divided by a \pa{planar} interface $\Gamma$ into two polygonal/polyhedral open subdomains $\Omega_S$ and $\Omega_D$, such that $\overline{\Omega}_S \cup \overline{\Omega}_D = \overline{\Omega}$ and $\Omega_S \cap \Omega_D = \emptyset$, cf. Figure~\ref{fig:Bidominio}. Here, $\Omega_S$ is the free-flow region modeled by Stokes equations, and $\Omega_D$ is the porous region modeled by Darcy's law. We split the boundary of the domain in $\Gamma_S = \partial\Omega_S \cap \partial\Omega$ and $\Gamma_D = \partial\Omega_D \cap \partial\Omega$. Moreover, we define $\bbf{n}_S$, $\bbf{n}_D$ the unit normal vectors to $\Gamma_S$ and $\Gamma_D$, respectively. We define $\bbf{n}_{\Gamma}^S$ the normal unit vector to the interface $\Gamma$ directed towards $\Omega_D$, and $\bbf{n}_{\Gamma}^D = - \bbf{n}_{\Gamma}^S$ the unit normal vector to the interface $\Gamma$ directed towards $\Omega_S$. Finally, we define an orthogonal system of unit tangent vectors $\bbf{t}_{\Gamma,j}$, $1\leq j \leq d - 1$ on $\Gamma$.
\begin{figure}[H]
        \centering
        \begin{tikzpicture}[x={(1cm,0cm)}, y={(0.4cm,0.3cm)}, z={(0cm,1cm)}]
\definecolor{myblue}{RGB}{220,232,245}
\definecolor{myred}{RGB}{242,224,228}
\definecolor{myorange}{RGB}{244,235,220}
\def\d{2}

\fill[myorange] (0,0,0) -- (3,0,0) -- (3,0,4) -- (0,0,4) -- cycle; 
\fill[myorange] (0,0,4) -- (3,0,4) -- (3,\d,4) -- (0,\d,4) -- cycle; 
\fill[myorange] (0,0,0) -- (0,\d,0) -- (0,\d,4) -- (0,0,4) -- cycle; 
\fill[myorange] (0,\d,0) -- (3,\d,0) -- (3,\d,4) -- (0,\d,4) -- cycle; 
\fill[myorange] (0,0,0) -- (3,0,0) -- (3,\d,0) -- (0,\d,0) -- cycle; 

\fill[myred] (3,0,0) -- (6,0,0) -- (6,0,4) -- (3,0,4) -- cycle; 
\fill[myred] (3,0,4) -- (6,0,4) -- (6,\d,4) -- (3,\d,4) -- cycle; 
\fill[myred] (6,0,0) -- (6,\d,0) -- (6,\d,4) -- (6,0,4) -- cycle; 
\fill[myred] (3,\d,0) -- (6,\d,0) -- (6,\d,4) -- (3,\d,4) -- cycle; 
\fill[myred] (3,0,0) -- (6,0,0) -- (6,\d,0) -- (3,\d,0) -- cycle; 

\fill[myblue] (3,0,0) -- (3,\d,0) -- (3,\d,4) -- (3,0,4) -- cycle;

\draw[gray!60] (0,\d,0) -- (6,\d,0) -- (6,\d,4) -- (0,\d,4) -- cycle;
\draw[gray!60] (3,\d,0) -- (3,\d,4);
\draw[gray!60] (3,0,0) -- (3,\d,0);
\draw[gray!60] (3,0,4) -- (3,\d,4);
\draw[thick] (0,0,0) -- (6,0,0) -- (6,\d,0) -- (0,\d,0) -- cycle;
\draw[thick] (0,0,4) -- (6,0,4) -- (6,\d,4) -- (0,\d,4) -- cycle;
\draw[thick] (0,0,0) -- (0,\d,0) -- (0,\d,4) -- (0,0,4) -- cycle;
\draw[thick] (6,0,0) -- (6,\d,0) -- (6,\d,4) -- (6,0,4) -- cycle;
\draw[thick] (0,0,0) -- (6,0,0) -- (6,0,4) -- (0,0,4) -- cycle;
\draw[thick] (3,0,0) -- (3,0,4);
\node at (1.5, 0, 1.5) {\Large $\Omega_S$};
\node at (4.5, 0, 1.5) {\Large $\Omega_D$};
\node[anchor=east] at (3.3, 0, -0.4) {\large $\Gamma$};
\node[anchor=south] at (3.4, 0, 4.7) {\LARGE $\Omega$};
\filldraw (3, \d/2, 2) circle (2pt);
\draw[->, thick] (3, \d/2, 2) -- (4, \d/2, 2);
\node[anchor=south] at (4, \d/2, 2.1) {$\mathbf{n}_{\Gamma}^S$};
\draw[->, thick] (3, \d/2, 2) -- (2, \d/2, 2);
\node[anchor=south] at (2, \d/2, 2.1) {$\mathbf{n}_{\Gamma}^D$};
\end{tikzpicture}
        \caption{The computational domain.}\label{fig:Bidominio}
\end{figure}
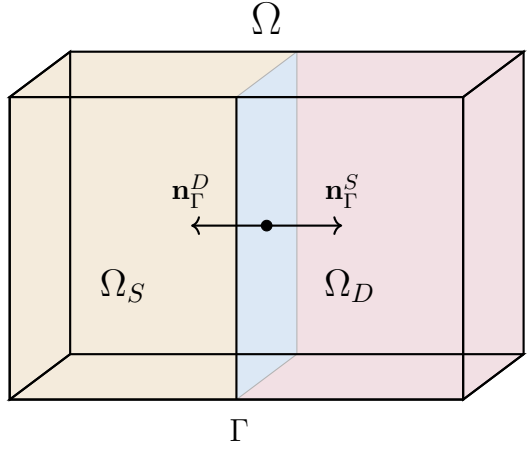 
The non-Newtonian Stokes--Darcy coupled system reads: 
given $\bbf{f}_S, \bbf{f}_D$, find $(\bbf{u}_S, p_S)$ and $(\bbf{u}_D, p_D)$ such that:
\begin{equation}\label{eq: strong problem}
\begin{cases}
    -\nabla \cdot (g_S(|\bbf{D}(\bbf{u}^S)|)\bbf{D}(\bbf{u}^S) - p^S\bbf{I}) = \bbf{f}_S & \text{in } \Omega_S, \\[6pt]
    \nabla \cdot \bbf{u}^S = 0 & \text{in } \Omega_S, \\[6pt]
    \bbf{u}^S = 0 & \text{on } \Gamma_S, \\[6pt]
    \bbf{K}^{-1}g_D(|\bbf{u}^D|)\bbf{u}^D + \nabla p^D = \bbf{f}_D & \text{in } \Omega_D, \\[6pt]
    \nabla \cdot \bbf{u}^D = 0 & \text{in } \Omega_D, \\[6pt]
    \bbf{u}^D \cdot \bbf{n}_D = 0 & \text{on } \Gamma_D, \\[6pt]
    \bbf{u}^S\cdot\bbf{n}_{\Gamma}^S + \bbf{u}^D\cdot\bbf{n}_{\Gamma}^D = 0 & \text{on } \Gamma, \\[6pt]
    p^S - (g_S(|\bbf{D}(\bbf{u}^S)|)\bbf{D}(\bbf{u}^S)\bbf{n}_{\Gamma}^S)\cdot\bbf{n}_{\Gamma}^S = p^D & \text{on } \Gamma, \\[6pt]
    \bbf{u}^S\cdot\bbf{t}_{\Gamma,j} + \rho (g_S(|\bbf{D}(\bbf{u}^S)|)\bbf{D}(\bbf{u}^S)\bbf{n}_{\Gamma}^S)\cdot\bbf{t}_{\Gamma,j} = 0 & \text{on } \Gamma, \text{ for all } 1 \leq j \leq d-1
\end{cases}
\end{equation}
where $\bbf{D}(\bbf{u}) = (\nabla \bbf{u} + \nabla^T\bbf{u})/2$ is the symmetric gradient of the velocity $\bbf{u}$, $g_S(|\bbf{D}(\bbf{u}^S)|)$ and $g_D(|\bbf{u}^D|)$ are the velocity-dependent viscosities, $\bbf{K}$ is a symmetric, positive definite tensor representing the permeability of the porous medium, and $\rho>0$ is a frictional parameter. We state the following assumptions on the velocity-dependent viscosities $g_S$ and $g_D$.
\begin{assumption} \label{assu: condition on g1 and g2}
    We assume that $g_S(\cdot)$ and $g_D(\cdot)$ are positive and bounded, and that $g_S(|\bbf{A}|)\bbf{A}$ and $g_D(|\bbf{a}|)\bbf{a}$ are Lipschitz continuous and strongly monotone, i.e., there exist positive constants $\underline{\nu}_S, \overline{\nu}_S, \underline{\mu}_S$, $\widetilde{\mu}_S$, $\underline{\nu}_D, \overline{\nu}_D, \underline{\mu}_D$ and $\widetilde{\mu}_D$ such that, for any symmetric $\bbf{A},\bbf{B} \in\mathbb{R}^{d\times d}$ and for all $\bbf{a},\bbf{b} \in\mathbb{R}^d$:
\begin{align}
    &\underline{\mu}_S\leq g_S(|\bbf{A}|)\leq \widetilde{\mu}_S, \label{eq: g_S positive}\\
    &|g_S(|\bbf{A}|)\bbf{A} - g_S(|\bbf{B}|)\bbf{B}|\leq \overline{\nu}_S|\bbf{A} - \bbf{B}|, \label{eq: g_S bdd and Lip continuous}\\
    &(g_S(|\bbf{A}|)\bbf{A} - g_S(|\bbf{B}|)\bbf{B}, \bbf{A} - \bbf{B}) \geq \underline{\nu}_S |\bbf{A}-\bbf{B}|^2. \label{eq: g_S strongly monotone}\\
     &\underline{\mu}_D\leq g_D(|\bbf{a}|)\leq \widetilde{\mu}_D, \label{eq: g_D positive}\\
    &|g_D(|\bbf{a}|)\bbf{a} - g_D(|\bbf{b}|)\bbf{b}|\leq \overline{\nu}_D |\bbf{a} - \bbf{b}|, \label{eq: g_D bdd and Lip continuous}\\
    &(g_D(|\bbf{a}|)\bbf{a} - g_D(|\bbf{b}|)\bbf{b}, \bbf{a} - \bbf{b}) \geq \underline{\nu}_D |\bbf{a}-\bbf{b}|^2. \label{eq: g_D strongly monotone}
\end{align}
\end{assumption}

We consider the following non-Newtonian models for the viscosity \cite{ChowCarey, ErvinJenkinsSunFEM}.
We denote by $g_S(|\bbf{D}(\bbf{u})|)$ the dynamic viscosity of the fluid defined as:
\begin{equation}
    g_S(|\bbf{D}(\bbf{u})|) = \underline{\mu}_S + (\overline{\mu}_S - \underline{\mu}_S) G(\sqrt{2}|\bbf{D}(\bbf{u})|) \qquad \text{where } \overline{\mu}_S > \underline{\mu}_S > 0 ,
\end{equation}
while the effective viscosity $g_D(\cdot)$ for the non-Newtonian porous media flow is defined as:
\begin{equation}
    g_D(|\bbf{u}|) = \underline{\mu}_D + (\overline{\mu}_D - \underline{\mu}_D) G(\sqrt{2}|\bbf{u}|) \qquad \text{where }\overline{\mu}_D > \underline{\mu}_D > 0.
\end{equation}
The choices of $G(s)$ are listed in Table \ref{tab:non-newtonian flows}.
\begin{table}[H]
\centering
\caption{Non-Newtonian models \cite{SDFDG}.}\label{tab:non-newtonian flows}
\begin{tabular}{|l|c|l|}
\hline
\textbf{Model} & $\bbf{G(s)}$ &   \\
\hline
Power & $s^{r-2}$ & $1 < r < 2$  \\
Carreau & $(\delta+s^2)^{\frac{r-2}{2}}$ & $1 \leq r < 2$ , $\delta>0$ \\
Carreau--Yasuda & $(\delta+s^a)^{\frac{r-2}{a}}$ & $a > 0, \ 1 \leq r < 2$, $\delta>0$  \\
Generalized cross & $(\delta+s^r)^{a}$ & $r > 0, \ a < 0, \ ar + 1 \geq0$, $\delta>0$  \\
Powell--Eyring & $\sinh{(s)}^{-1}/s$ &   \\
\hline
\end{tabular}
\end{table}

\begin{rmk}
    If both the viscosities from $g_S(\cdot)$ and $g_D(\cdot)$ obey a Carreau--Yasuda model, i.e.:
    \begin{equation}
        g_S(|\bbf{D}(\bbf{u})|) = \underline{\mu}_S + (\overline{\mu}_S - \underline{\mu}_S) (\delta + |\bbf{D}(\bbf{u})|^2)^{\frac{r-2}{2}} \ \text{  and  } \ g_D(|\bbf{u}|) = \underline{\mu}_S + (\overline{\mu}_S - \underline{\mu}_S) (\delta+  |\bbf{u}|^2)^{\frac{r-2}{2}}
    \end{equation}
    then Assumption~\ref{assu: condition on g1 and g2} is satisfied provided that $\delta>0$ and $r \geq 1$, cf. \cite{SDFDG}. 
\end{rmk}
Finally, we assume that $k_{\min}|\bm{\xi}|^2\leq\bm{\xi}^T\bbf{K}\bm{\xi}\leq k_{\max}|\bm{\xi}|^2$ for all $\bm{\xi}\in\mathbb{R}^d$.\\

Throughout the paper, we make use of the following notation: $a \lesssim b$ meaning that $a \leq C b$, with $C$ a constant independent of the discretization parameter $h$ but that might depend on the $\overline{\mu}_S, \underline{\mu}_S, \overline{\mu}_D, \underline{\mu}_D, \overline{\nu}_S, \underline{\nu}_S, \overline{\nu}_D, \underline{\nu}_D$, the domain geometry and the polynomial degree.\\

We conclude this section by introducing the weak formulation of \eqref{eq: strong problem}. 
We also extend the continuous inf-sup condition for the Stokes problem to the coupled problem, using arguments based on space inclusions. First, we define the function spaces:
\begin{align*}
    &\bbf{X}^S = \{ \bbf{v}^S \in [H^1(\Omega_S)]^d: \bbf{v}^S = 0 \text{ on }\Gamma_S \} = [H^1_{0,\Gamma_S}(\Omega_S)]^d, \qquad M^S = L^2(\Omega_S), \\
    &\bbf{X}^D = \{ \bbf{v}^D \in [L^2(\Omega_D)]^d: \nabla \cdot \bbf{v}^D \in L^2(\Omega_D), \ \bbf{v}^D\cdot\bbf{n}_{\Gamma}^D = 0 \text{ on } \Gamma_D \}, \qquad M^D = L^2(\Omega_D), \\
    &\bbf{X} = \{(\bbf{v}^S, \bbf{v}^D) \in \bbf{X}^S \times \bbf{X}^D : \bbf{v}^S\cdot\bbf{n}_\Gamma^S = - \bbf{v}^D\cdot\bbf{n}_\Gamma^D \text{ on } \Gamma\}, \\ & M = \{(q^S,q^D)\in M^S\times M^D: (q^S,1)_{\Omega_S} = -(q^D,1)_{\Omega_D} \},
\end{align*}
and their norms:
\begin{flalign}
   & \| \bbf{v}^S \|_{\bbf{X}^S} = \| \bbf{v}^S \|_{H^1(\Omega_S)},
   \ \ \| p^S \|_{M^S} = \| p^S \|_{L^2(\Omega_S)}, \\
     & \| \bbf{v}^D \|_{\bbf{X}^D} = \| \bbf{v}^D \|_{L^2(\Omega_D)},
   \ \ \| p^D \|_{M^D} = \| p^D \|_{L^2(\Omega_D)}, \\
    &\| \bbf{v} \|_\bbf{X}= \| \bbf{v}^S \|_{\bbf{X}^S} +  \| \bbf{v}^D \|_{\bbf{X}^D}, 
   \ \ \| p \|_M = \| p^S \|_{L^2(\Omega_S)} + \| p^D \|_{L^2(\Omega_D)},
\end{flalign}
The weak form of \eqref{eq: strong problem} reads: find $\bbf{u} = (\bbf{u}^S,\bbf{u}^D) \in \bbf{X}$ and $p = (p^S,p^D) \in M$ such that
\begin{equation}\label{eq: weak SDF}
    \begin{cases}
        a(\bbf{u},\bbf{v}) + b(\bbf{v},p) = (\bbf{f},\bbf{v}) \ \ \ & \forall\bbf{v}\in \bbf{X},\\
        b(\bbf{u},q) = 0 \ \ \ &\forall q \in M,
    \end{cases}
\end{equation}
where
\begin{align}
    &a_S(\bbf{u}^S, \bbf{v}^S) = \int_{\Omega_S} g_S(|\bbf{D}(\bbf{u}^S)|)\bbf{D}(\bbf{u}^S):\bbf{D}(\bbf{v}^S)\ dx + \sum_{j=1}^{d-1} \int_\Gamma \rho^{-1}(\bbf{u}^S\cdot\bbf{t}_{\Gamma,j})(\bbf{v}^S\cdot\bbf{t}_{\Gamma,j})\ ds,\\
    &a_D(\bbf{u}^D, \bbf{v}^D) =  \int_{\Omega_D} \bbf{K}^{-1}g_D(|\bbf{u}^D|)\bbf{u}^D \cdot\bbf{v}^D \ dx, \\
    &a(\bbf{u},\bbf{v}) = a_S(\bbf{u}^S, \bbf{v}^S) + a_D(\bbf{u}^D, \bbf{v}^D), \\ 
    &b_S(\bbf{v}^S, p^S) = - \int_{\Omega_S} q^S\nabla\cdot\bbf{v}^S \ dx,\\
    &b_D(\bbf{v}^D, p^D) = - \int_{\Omega_D} q^D\nabla\cdot\bbf{v}^D \ dx,\\
    &b(\bbf{v}, p) = b_S(\bbf{v}^S, p^S) + b_D(\bbf{v}^D, p^D),\\
    &(\bbf{f}, \bbf{v}) = \int_{\Omega_S} \bbf{f}_S \cdot \bbf{v}^S dx + \int_{\Omega_D} \bbf{f}_D \cdot \bbf{v}^D dx.
\end{align}

\subsection{Well-posedness and stability}
To prove the well-posedness of problem \eqref{eq: weak SDF}, we first prove an inf-sup condition in the continuous setting by considering the inf-sup condition of the Stokes problem, and extending it to the coupled system.
\begin{theorem}[Inf-sup condition]
    For all $p=(p^S,p^D)\in M$ it holds
    \begin{equation}\label{thm: continuous infsup}
        \sup_{\bbf{v}\in \bbf{X}, \bbf{v} \neq \bbf{0}}\frac{b(\bbf{v}, p)}{\|\bbf{v}\|_{\bbf{X}}} \gtrsim \|p\|_{M}.
    \end{equation}
\end{theorem}
\begin{proof}
    First, we observe that $[H^1_{0,\Gamma_{D}}(\Omega_D)]^d\subset \bbf{X}^D$,
    therefore, for all $\bbf{v}\in [H^1_{0,\partial\Omega}(\Omega)]^d$, one has
    $(\bbf{v}^S,\bbf{v}^D)=(\bbf{v}_{|\Omega_S},\bbf{v}_{|\Omega_D})$ is such that $\bbf{v}^S\cdot\bbf{n}_\Gamma^S = \bbf{v}^D\cdot\bbf{n}_\Gamma^D$. As a result, we have
    $[H^1_{0,\partial\Omega}(\Omega)]^d\subset \bbf{X}$
    which also implies that
    \begin{equation}\label{eq:continuous norm dis}
        \|\cdot\|_\bbf{X} \lesssim \|\cdot\|_{H^1(\Omega)}.
    \end{equation}
    Thus, considering the standard Stokes inf-sup condition due to the fact that $p\in M$ has zero average over $\Omega$, it is inferred that
    \begin{equation}\label{eq: stokes continuous infsup}
        \|p\|_{L^2(\Omega)} \lesssim \sup_{\bbf{v}\in [H^1_{0,\partial\Omega}(\Omega)]^d} \frac{b(\bbf{v}, p)}{\|\bbf{v}\|_{H^1(\Omega)}}
        \leq \sup_{\bbf{v}\in \bbf{X}, \bbf{v} \neq 0}
        \frac{b(\bbf{v}, p)}{\|\bbf{v}\|_{H^1(\Omega)}}
        \qquad \forall p \in M.
    \end{equation}
    since we are taking the superior on a larger space.
    Additionally, from \eqref{eq:continuous norm dis} we get 
    \begin{equation}\label{eq: continuous infsup fin}
        \sup_{\bbf{v}\in \bbf{X}, \bbf{v} \neq 0} \frac{b(\bbf{v}, p)}{\|\bbf{v}\|_{H^1(\Omega)}} \lesssim \sup_{\bbf{v}\in \bbf{X}, \bbf{v} \neq 0} \frac{b(\bbf{v}, p)}{\|\bbf{v}\|_{\bbf{X}}} \qquad \forall p \in M.
    \end{equation}
    Hence, putting together \eqref{eq: continuous infsup fin} and \eqref{eq: stokes continuous infsup} we obtain \eqref{thm: continuous infsup}.
\end{proof}

The well-posedness of \eqref{eq: weak SDF} is a consequence of \eqref{eq: continuous infsup fin} and of the boundedness and coercivity of the nonlinear form $a(\cdot, \cdot)$.
We denote the kernel of the bilinear form $b$ in \eqref{eq: weak SDF} by
\begin{equation*}
    \mathring{\bbf{X}} = \{\bbf{v}\in \bbf{X} : \nabla\cdot\bbf{v}^S = 0 \text{ a.e. on } \Omega_S, \nabla\cdot\bbf{v}^D = 0 \text{ a.e. on } \Omega_D \},
\end{equation*}
and its dual by $\mathring{\bbf{X}}^*$. We define the operator $A: \mathring{\bbf{X}}\to\mathring{\bbf{X}}^*$ as
\begin{equation}\label{eq: A defintion}
    (A(\bbf{u}), \bbf{v})=a(\bbf{u}, \bbf{v}) \ \ \ \forall\bbf{u},\bbf{v}\in \mathring{\bbf{X}}.
\end{equation}
We next show the following result for the problem
$A(\bbf{u})= \bbf{f}$ in $\mathring{\bbf{X}}^*$, which corresponds to the restriction of \eqref{eq: weak SDF} to the kernel  $\mathring{\bbf{X}}$.

\begin{proposition}
    Let Assumption~\ref{assu: condition on g1 and g2} be satisfied. Then, the operator $A: \mathring{\bbf{X}}\to\mathring{\bbf{X}}^*$ is uniformly bounded, i.e.
    \begin{equation}\label{eq: A bound}
        \|A(\bbf{u})\|_{\mathring{\bbf{X}}^*} \lesssim \|\bbf{u}^S\|_{\bbf{X}^S} + \|\bbf{u}^D\|_{L^2(\Omega_D)},
    \end{equation}
    where the hidden constants depend on $\widetilde{\mu}_S$, $\widetilde{\mu}_D$ and $\rho$.
\end{proposition}
\begin{proof}
    Owing to the definition of dual norm and \eqref{eq: A defintion}, we obtain
    \begin{equation}\label{eq: dualnorm}
        \|A(\bbf{u})\|_{\mathring{\bbf{X}}^*} = \sup_{\bbf{v}\in\mathring{\bbf{X}}, \bbf{v} \neq \bbf{0}} \frac{(A(\bbf{u}), \bbf{v})}{\|\bbf{v}\|_{\mathring{\bbf{X}}}} = \sup_{\bbf{v}\in\mathring{\bbf{X}}, \bbf{v} \neq \bbf{ 0}} \frac{a_S(\bbf{u}^S, \bbf{v}^S) + a_D(\bbf{u}^D, \bbf{v}^D)}{\|\bbf{v}\|_{\mathring{\bbf{X}}}}.
    \end{equation}
    By applying \eqref{eq: g_S positive}, \eqref{eq: g_D positive}, the Cauchy--Schwarz inequality, and the trace inequality, we obtain 
    \begin{align*}
        &a_S(\bbf{u}^S, \bbf{v}^S) \lesssim\|\bbf{D}(\bbf{u}^S)\|_{L^2(\Omega_S)}\|\bbf{D}(\bbf{v}^S)\|_{L^2(\Omega_S)} + \|\bbf{u}^S\|_{L^2(\Gamma)}\|\bbf{v}^S\|_{L^2(\Gamma)} \lesssim \|\bbf{u}^S\|_{\bbf{X}^S}\|\bbf{v}^S\|_{\bbf{X}^S},\\
        &a_D(\bbf{u}^D, \bbf{v}^D) \lesssim \|\bbf{u}^D\|_{L^2(\Omega_D)}\|\bbf{v}^D\|_{L^2(\Omega_D)} \leq \|\bbf{u}^D\|_{L^2(\Omega_D)}\|\bbf{v}^D\|_{\bbf{X}^D},
    \end{align*}
    with constants depending on $\widetilde{\mu}_S, \widetilde{\mu}_D$ and $\rho$.
    Plugging these inequalities into \eqref{eq: dualnorm}, we have \eqref{eq: A bound}.
\end{proof}

\begin{proposition}
    Let Assumption~\ref{assu: condition on g1 and g2} be satisfied; then, the operator A is coercive, i.e.,
    \begin{equation}\label{eq: A coercive bound}
        (A(\bbf{v}), \bbf{v}) \gtrsim \|\bbf{v}^S\|^2_{\bbf{X}^S} + \sum_{j=1}^{d-1} \| \bbf{v}^S \cdot \bbf{t}_{\Gamma,j} \|_{L^2(\Gamma)}^2 +  \|\bbf{v}^D\|^2_{L^2(\Omega_D)} \ \forall\bbf{v}\in\mathring{\bbf{X}},
    \end{equation}
    where the hidden constant depend on $\underline{\mu}_S, \underline{\mu}_D$, $\rho$, and $k_{max}$.
\end{proposition}
\begin{proof}
    Applying Korn's first inequality, the Poincaré inequality, and the fact that $\nabla\cdot\bbf{v}^D = 0$, since $\bbf{v} = (\bbf{v}^S, \bbf{v}^D) \in \mathring{\bbf{X}}$, we get
    \begin{equation}
    \begin{split}
        (A(\bbf{v}), \bbf{v}) = a(\bbf{v}, \bbf{v}) &\gtrsim \|\bbf{D}(\bbf{v}^S)\|^2_{L^2(\Omega_S)} + \sum_{j=1}^{d-1}\rho^{-1} \| \bbf{v}^S \cdot \bbf{t}_{\Gamma,j}\|^2_{L^2(\Gamma)} +\|\bbf{v}^D\|^2_{L^2(\Omega_D)}\\
        &\gtrsim \|\bbf{v}^S\|^2_{\bbf{X}^S} + \sum_{j=1}^{d-1}\rho^{-1} \| \bbf{v}^S \cdot \bbf{t}_{\Gamma,j}\|^2_{L^2(\Gamma)} + \|\bbf{v}^D\|^2_{L^2(\Omega_D)},
    \end{split}
    \end{equation}
    and the proof is complete.
\end{proof}
We next show that the operator $A$ defined as in \eqref{eq: A defintion} is strongly monotone.
\begin{proposition}
    Under Assumption~\ref{assu: condition on g1 and g2} the operator $A$ is strongly monotone, i.e., for all $\bbf{u}, \bbf{v} \in \mathring{\bbf{X}}$,
    \begin{equation}\label{eq: A monotone}
        (A(\bbf{u}) - A(\bbf{v}), \bbf{u} - \bbf{v}) \gtrsim \|\bbf{u}^S - \bbf{v}^S\|^2_{\bbf{X}^S} + \sum_{j=1}^{d-1}\|(\bbf{u}^S - \bbf{v}^S) \cdot \bbf{t}_{\Gamma,j} \|_{L^2(\Gamma)}^2 + \|\bbf{u}^D - \bbf{v}^D\|^2_{L^2(\Omega_D)},
    \end{equation}
    where the hidden constants depend on $\underline{\nu}_S$, $\underline{\nu}_D$, $\rho$, and $k_{max}$.
\end{proposition}
\begin{proof}
    Let $\bbf{u} = (\bbf{u}^S, \bbf{u}^D),\bbf{v} = (\bbf{v}^S, \bbf{v}^D)\in\mathring{\bbf{V}}$. We write:
    \begin{equation}
        (A(\bbf{u}) - A(\bbf{v}),\bbf{u} - \bbf{v}) = \underbrace{a_S(\bbf{u}^S, \bbf{u}^S - \bbf{v}^S) - a_S(\bbf{v}^S, \bbf{u}^S - \bbf{v}^S)}_{I_1} + \underbrace{a_D(\bbf{u}^D, \bbf{u}^D - \bbf{v}^D) - a_D(\bbf{v}^D, \bbf{u}^D - \bbf{v}^D)}_{I_2} .
    \end{equation}
By \eqref{eq: g_S strongly monotone}, the Korn's inequality and the Poincaré inequality, we have:
\begin{equation}
    I_1 \gtrsim \|\bbf{D}(\bbf{u}^S - \bbf{v}^S)\|^2_{L^2} + \sum_{j=1}^{d-1}\rho^{-1} \| (\bbf{u}^S - \bbf{v}^S) \cdot \bbf{t}_{\Gamma,j} \|_{L^2(\Gamma)}^2 \gtrsim 
    \|(\bbf{u}^S - \bbf{v}^S)\|^2_{\bbf{X}^S} + \sum_{j=1}^{d-1} \rho^{-1} \| (\bbf{u}^S - \bbf{v}^S) \cdot \bbf{t}_{\Gamma,j} \|_{L^2(\Gamma)}^2.
\end{equation}
Additionally, \eqref{eq: g_D strongly monotone} implies:
\begin{equation}
    I_2 \gtrsim \|\bbf{u}^D - \bbf{v}^D\|^2_{L^2(\Omega_D)}.
\end{equation}
Hence, we can conclude \eqref{eq: A monotone}.
\end{proof}
Combining the previous results, we obtain the existence and uniqueness of the solution to  $A(\bbf{u})= \bbf{f}$ in $\mathring{\bbf{X}}^*$. Thanks to the inf-sup condition \eqref{thm: continuous infsup}, we can conclude the well-posedness of problem \eqref{eq: weak SDF}.\\

We conclude the section by deriving a stability estimate. We consider problem \eqref{eq: weak SDF}, and we take $\bbf{v} = \bbf{u}$ and $q = p$. Summing up the two equations we obtain:
\begin{equation}
    a(\bbf{u}, \bbf{u}) = (\bbf{f}, \bbf{u}) \lesssim \| \bbf{f}_S \|_{L^2(\Omega_S)}^2 + \| \bbf{u}^S \|_{L^2(\Omega_S)}^2 + \| \bbf{f}_D \|^2_{L^2(\Omega_D)} + \| \bbf{u}^D \|^2_{L^2(\Omega_D)}.
\end{equation}
Due to the coercivity of the form $a(\cdot, \cdot)$, see \eqref{eq: A coercive bound}, we obtain the following:
\begin{equation}
    \| \bbf{u}^S \|^2_{\bbf{X}^S} + \sum_{j=1}^{d-1} \rho^{-1} \| \bbf{u}^S \cdot \bbf{t}_{\Gamma,j} \|_{L^2(\Gamma)}^2 + \| \bbf{u}^D \|_{\bbf{X}^D}^2  \lesssim \| \bbf{f}_S \|_{L^2(\Omega_S)}^2 + \| \bbf{f}_D \|^2_{L^2(\Omega_D)}.
\end{equation}
In order to prove a stability estimate for the pressure we consider the inf-sup condition \eqref{thm: continuous infsup}
\begin{equation}
\begin{split}
    \|p\|_{L^2(\Omega)} &\lesssim \sup_{\bbf{v} \in H^1_{0,\Gamma}} \frac{b(\bbf{v}, p)}{\|\bbf{v}\|_\bbf{X}} = \sup_{\bbf{v} \in H^1_{0,\Gamma}} \frac{(\bbf{f}, \bbf{v}) - a(\bbf{u}, \bbf{v})}{\|\bbf{v}\|_\bbf{X}} \\ & \lesssim \|\bbf{u}^S\|_{\bbf{X}^S} + \sum_{j=1}^{d-1} \rho^{-1}\|\bbf{u}^S\cdot\bbf{t}_{\Gamma,j}\|_{L^2(\Gamma)}+ \|\bbf{f}_S\|_{L^2(\Omega_S)} + \| \bbf{u}^D \|_{L^2(\Omega_D)} + \| \bbf{f}_D \|_{L^2(\Omega_D)}.
    \end{split}
\end{equation}
Therefore, from the last two inequalities we derive
$
    \| p \|_{L^2(\Omega)} \lesssim \|\bbf{f}_S\|_{L^2(\Omega_S)} + \|\bbf{f}_D\|_{L^2(\Omega_D)}.
$

\section{Discrete problem} \label{sec: discrete problem}
In this section, we introduce the discrete framework of the PolyDG method, some functional analysis tools \cite{Leray-Lions, HighOrderpolynomialapprox, SobolevEmbedding}, and we introduce the discrete formulation of the problem.\\

We introduce a polytopal mesh partition $\scr{K}_h = \scr{K}_{h,\mcal{S}} \cup \scr{K}_{h,\mcal{D}}$, of the domain $\Omega$, with the set of the internal and boundary faces $\scr{F}_h$, where $\scr{K}_{h,\mcal{I}} = \{ K\in\scr{K}_{h,\mcal{I}} : \Bar{K}\subseteq\Bar{\Omega}_\mcal{I} \}$, for $\mcal{I} = \mcal{S},\mcal{D}$. We further assume that $\scr{K}_{h,\mcal{I}}$, for $\mcal{I} = \mcal{S},\mcal{D}$, are aligned with the subdomains $\Omega_\mcal{I}$. We denote by $|K|$ the measure of the element $K \in \scr{K}_h$ and by $h_K$ its diameter. Additionally, we set $h = \max_{K \in \scr{K}_h} h_K < 1$. We define the interface as the intersection of the $(d - 1)$-dimensional facets of two neighboring elements, and we distinguish two cases. When $d = 2$, the faces are always line segments, and we denote their set as $\scr{F}_{h,\mcal{I}}$, with $\mcal{I} = \mcal{S},\mcal{D}$. When $d = 3$, the faces are generic polygons; we further assume that we can decompose any such face into planar triangles. We denote the set of such triangles as $\scr{F}_{h,\mcal{I}}$, with $\mcal{I} = \mcal{S},\mcal{D}$. We decompose $\scr{F}_{h,\mcal{I}}$ into the set of the internal faces, $\scr{F}_{h,\mcal{I}}^i$, the boundary faces, $\scr{F}_{h,\mcal{I}}^b$, lying on the boundary $\partial \Omega_\mcal{I}\setminus\Gamma$, and of the internal faces $\Gamma_h$, lying on the interface $\Gamma$ between the subdomains $\Omega_\mcal{S}$ and $\Omega_\mcal{D}$. Moreover, we can split the boundary faces accordingly to the type of boundary condition imposed: $\scr{F}_{h,\mcal{I}}^b = \scr{F}_{h,\mcal{I}}^d \cup \scr{F}_{h,\mcal{I}}^n$, where $\scr{F}_{h,\mcal{I}}^d $ and $ \scr{F}_{h,\mcal{I}}^n$, are the boundary faces where we impose Dirichlet and Neumann boundary conditions respectively. 

\begin{assumption} \label{assump: mesh properties} \cite{HighOrderpolynomialapprox}
    We denote by $\mathcal{H}\subset (0,+\infty)$ a countable set of meshsizes having $0$ as its unique accumulation point. A family of meshes $(\scr{M}_h)_{h\in\mcal{H}} = (\scr{K}_h, \scr{F}_h)_{h\in\mcal{H}}$ is said to be regular if there exists a real number $\gamma \in (0,1)$ independent of $h$ such that, for all $h \in \mcal{H}$, there exists a matching simplicial submesh, in the sense of \cite{HighOrderpolynomialapprox}, $\widetilde{\scr{M}}_h = (\widetilde{\scr{K}}_h, \widetilde{\scr{F}}_h)$ of $\scr{M}_h$ that satisfies the following conditions:
    \begin{itemize}
        \item[1.] \textbf{Shape Regularity}: for any simplex $\tau \in \widetilde{\scr{K}}_h$, denoting by $h_\tau$ its diameter and by $r_\tau$ its inradius, it holds $\gamma h_\tau \leq r_\tau$;
        \item[2.] \textbf{Contact Regularity}: for any mesh element $K \in \scr{K}_h$ and any simplex $\tau \in \widetilde{\scr{K}}_h$ where $\widetilde{\scr{K}}_h = \{ \tau \in \widetilde{\scr{K}}_h: \tau \subset K\}$ is the set of the simplices contained in $K$, it holds $\gamma h_K \leq h_\tau$.

    \end{itemize}
\end{assumption}
From now on, we will omit the subscript $h$ to simplify the notation.
Given a polynomial degree $l \geq 0$ and a polytopal mesh $\scr{M} = (\scr{K}, \scr{F})$, we denote by $P^l(K)$ the space of polynomials of degree $l$ on the element $K\in \scr{K}$, by 
$P^l(\scr{K})$ the set of piecewise polynomials of order $l$ over $\scr{K}$.
Under Assumption~\ref{assump: mesh properties}, the following trace-inverse inequality holds
\begin{equation}\label{lemm: trace inverse inequality}
    \|v_h\|_{L^2(\partial K)} \lesssim h_K^{-\frac{1}{2}}\|v_h\|_{L^2(K)} \qquad \forall K \in \scr{K} \ \forall v_h \in P^l(K),
\end{equation}
where the hidden constant depends on the polynomial degree $l$, the space dimension $d$, and the mesh regularity parameter $\gamma$.
We define the following broken functional spaces:
\begin{align}
    &\bbf{X}^\mcal{I}_h = \{\bbf{v}^\mcal{I}_h \in [L^2(\Omega_\mcal{I})]^d:\bbf{v}^\mcal{I}_h|_K \in [P^l(K)]^d \ (\forall K \in \scr{K}_\mcal{I})\}, \\
    &\bbf{X}_h = \bbf{X}_h^S \times \bbf{X}_h^D,\\
    &M^\mcal{I}_h = \{q_\mcal{I}^h \in M^\mcal{I} : q^\mcal{I}_h|_K \in P^{l}(K) \ (\forall K \in \scr{K}_\mcal{I}) \}, \\
    &M_h=\{(\ph{q}{S},\ph{q}{D}) \in M^S_h\times M^D_h:(\ph{q}{S},1)_{\Omega_S} + (\ph{q}{D},1)_{\Omega_D}=0 \},
\end{align}
with $\mcal{I} = S,D$.
Now we introduce the average operator $\averagel \cdot \averager$ and the jump operators $\jumpl \cdot \jumpr$ and $\jumpl \cdot \jumpr_\bbf{n}$ \cite{BrezziMarini},for the scalar and vector quantities $\psi$ and $\bm{\varphi}$, on each internal facet $F\in \partial K_1 \cup \partial K_2$:
\begin{align*}
    &\averagel\psi\averager =\frac{1}{2}(\psi|_{K_1} + \psi|_{K_2}), \ \ \averagel \bm{\varphi} \averager =\frac{1}{2}(\bm{\varphi}|_{K_1} + \bm{\varphi}|_{K_2}), \\ & \jumpl \psi \jumpr = \psi|_{K_1} - \psi|_{K_2}, \ \ \jumpl \bm{\varphi} \jumpr = \bm{\varphi}|_{K_1} -  \bm{\varphi}|_{K_2}, \\ & \jumpl \psi \jumpr_\bbf{n} = \psi|_{K_1}\bbf{n}|_{K_1} + \psi|_{K_2} \bbf{n}|_{K_2}, \ \ \jumpl \bm{\varphi} \jumpr_\bbf{n} = \bm{\varphi}|_{K_1}\cdot\bbf{n}|_{K_1} +   \bm{\varphi}|_{K_2}\cdot\bbf{n}|_{K_2}.
\end{align*}
If $F \in K_1$, with $F$ a face on the Dirichlet boundary, we set $\averagel \psi \averager = \psi|_{K_1}$, $\averagel \bm{\varphi} \averager = \bm{\varphi}|_{K_1}$. Additionally, we define the trace operators for the trial functions on the faces $F\in\scr{F}^b$ as $\jumpl \psi \jumpr = \psi$, $ \jumpl \bm{\varphi} \jumpr = \bm{\varphi}$, $\jumpl \psi \jumpr_\bbf{n} = \psi\bbf{n}$ and $ \jumpl \bm{\varphi} \jumpr_\bbf{n} = \bm{\varphi}\cdot\bbf{n}$. 

We define the following penalty functions: $\sigma_\mcal{I}: \scr{F} \to \mathbb{R}$ and $\xi_\mcal{I}: \scr{F}^i \to \mathbb{R}$ as:
\begin{equation}\label{eq: penalty}
\sigma_\mcal{I}\big|_F =
\begin{cases}
\gamma^v_\mcal{I} \max\limits_{K^+,K^-} \left\{\frac{l^2}{h_K} \right\}
& F \in \mathcal{F}_\mcal{I}^i, \\[1.2ex]

\gamma^v_\mcal{I} \dfrac{l^2}{h_K}
& F \in \mathcal{F}_\mcal{I}^b \cup \Gamma ,
\end{cases}
\qquad
\xi_\mcal{I}\big|_F =
\gamma^p_\mcal{I} \min\limits_{K^+,K^-}\left\{ \frac{h_K}{l} \right\}
\qquad
F \in \mathcal{F}_\mcal{I}^i, 
\end{equation}
with $\mcal{I} = S,D,\Gamma$ and $\gamma_\mcal{I}^v$ and $\gamma_\mcal{I}^p$ are user-dependent penalty parameters.

Now, we define the following discrete forms: 
\begin{align}
    \begin{split}
        &\ah{\vh{u}{S}}{\vh{v}{S}}{S} = \sum_{K\in\scr{K}_\mcal{S}} \int_K g_S(|\bbf{D}(\vh{u}{S})|)\bbf{D}(\vh{u}{S}):\bbf{D}(\vh{v}{S})\ dx \ - \sum_{F\in\scr{F}_\mcal{S}} \int_F \averagel g_S(|\bbf{D}(\vh{u}{S})|)\bbf{D}(\vh{u}{S})\bbf{n} \averager \cdot \jumpl \vh{v}{S} \jumpr \ ds \\& 
        \qquad\qquad\qquad+ \sum_{j=1}^{d-1} \sum_{F\in\Gamma} \int_F \rho^{-1}(\vh{u}{S}\cdot\bbf{t}_{\Gamma,j})(\vh{v}{S}\cdot\bbf{t}_{\Gamma,j}) \ ds \ + 
            \sum_{F\in\scr{F}_\mcal{S}} \int_F \sigma_S\jumpl \vh{u}{S}\jumpr \cdot \jumpl \vh{v}{S} \jumpr \ ds,
    \end{split}\\
    \begin{split}
             &\ah{\vh{u}{D}}{\vh{v}{D}}{D} = \sum_{K\in\scr{K}_\mcal{D}} \int_K \bbf{K}^{-1}g_D(|\vh{u}{D}|)\vh{u}{D}\cdot\vh{v}{D} \ dx + \sum_{F\in\scr{F}_\mcal{D}^i} \int_F \sigma_D \jumpl \vh{u}{D} \jumpr_\bbf{n} \jumpl \vh{v}{D} \jumpr_\bbf{n} \ ds.
        \end{split}\\
\end{align}
Since in the analysis we need to impose the continuity of the normal jump of the velocities on $\Gamma$, we add a stabilization term:
\begin{equation}
    \ah{\vh{u}{}}{\vh{v}{}}{\Gamma}=\sum_{F\in\Gamma} \int_F \sigma_\Gamma \jumpl \vh{u}{} \jumpr_\bbf{n} \jumpl \vh{v}{} \jumpr_\bbf{n} \ ds, 
\end{equation}
where $\vh{u}{} = (\vh{u}{S}, \vh{u}{D})$.
Hence, we set 
\begin{equation}
    \ah{\vh{u}{}}{\vh{v}{}}{} = \ah{\vh{u}{S}}{\vh{v}{S}}{S} + \ah{\vh{u}{D}}{\vh{v}{D}}{D} + \ah{\vh{u}{}}{\vh{v}{}}{\Gamma}.
\end{equation}
We also define the following bilinear forms
\begin{align}
        &\bh{\vh{v}{S}}{\ph{p}{S}}{S}= -\sum_{K\in\scr{K}_\mcal{S}} \int_K \ph{p}{S}\nabla\cdot \vh{v}{S} \ dx \ + \sum_{F\in\scr{F}_\mcal{S}} \int_F \averagel \ph{p}{S} \averager  \jumpl \vh{v}{S}\jumpr_\bbf{n} \ ds,\\
        &\bh{\vh{v}{D}}{\ph{p}{D}}{D} = - \sum_{K\in\scr{K}_\mcal{D}} \int_K \ph{p}{D}\nabla\cdot \vh{v}{D} \ dx \ + \sum_{F\in\scr{F}_\mcal{D}} \int_F \averagel \ph{p}{D}\averager  \jumpl \vh{v}{D} \jumpr_\bbf{n} \ ds,\\
        &\bh{\vh{v}{}}{\ph{p}{}}{\Gamma} = \sum_{F\in\Gamma}\int_F \ph{p}{D} \jumpl\vh{v}{} \jumpr_\bbf{n} ds.
\end{align}
and set $\bh{\vh{v}{}}{\ph{p}{}}{}$ as
\begin{equation}\label{eq: bh definition}
    \bh{\vh{v}{}}{\ph{p}{}}{} = \bh{\vh{v}{S}}{\ph{p}{S}}{S} + \bh{\vh{v}{D}}{\ph{p}{D}}{D} + \bh{\vh{v}{}}{\ph{p}{}}{\Gamma},
\end{equation}
where $\ph{p}{} = (\ph{p}{S}, \ph{p}{D})$.
Furthermore, we add the following stabilization term on the pressure:
\begin{equation}
    \sh{\ph{p}{}}{\ph{q}{}} = \sum_{F\in\scr{F}_\mcal{S}^i} \int_F \xi_S \jumpl \ph{p}{S}\jumpr_\bbf{n} \jumpl \ph{q}{S}\jumpr_\bbf{n} ds + \sum_{F\in\scr{F}_\mcal{D}^i} \int_F \xi_D \jumpl \ph{p}{D}\jumpr_\bbf{n} \jumpl \ph{q}{D}\jumpr_\bbf{n} ds
    + \sum_{F\in\Gamma} \int_F \xi_\Gamma \jumpl \ph{p}{}\jumpr_\bbf{n} \jumpl \ph{q}{}\jumpr_\bbf{n} ds.
\end{equation}
Finally, we define the following forcing term:
\begin{equation}
    F_h(\vh{v}{}) = (\bbf{f}_S,\vh{v}{S})_{\Omega_S} + (\bbf{f}_D, \vh{v}{D})_{\Omega_D}.
\end{equation}
The discrete problem reads: find $\bbf{u}_h=(\vh{u}{S},\vh{u}{D}) \in \bbf{X}_h$ and $\ph{p}{} = (\ph{p}{S}, \ph{p}{D}) \in M_h$ such that:
\begin{equation}\label{eq: discrete SDF}
    \begin{cases}
        a_h(\vh{u}{}, \vh{v}{}) + \bh{\vh{v}{}}{\ph{p}{}}{} = F(\vh{v}{}) \ \ \ & \forall\vh{v}{}\in \bbf{X}_h,\\
        \bh{\vh{u}{}}{\ph{q}{}}{} - \sh{\ph{p}{}}{\ph{q}{}} = 0 \ \ \ &\forall \ph{q}{} \in M_h.
    \end{cases}
\end{equation}

\section{Well-posedness of the discrete problem} \label{sec: well-posedness}
In this section, we analyze the well-posedness of the discrete problem \eqref{eq: discrete SDF}. First, we define some useful notation, prove some properties of the bilinear forms $\ah{\cdot}{\cdot}{}$ and $\bh{\cdot}{\cdot}{}$, and we show how we can extend the inf-sup condition for the Stokes problem to the coupled problem. 
Then we prove the existence and uniqueness of the solution, showing that the form associated with the complete problem is weakly continuous and maps linearly into a coercive space. Finally, we show a continuous dependence of the problem on the data.\\

Now, we introduce the discrete norms that are used in the \emph{a-priori} analysis of the PolyDG scheme. 
\begin{definition}[Discrete norms]
\label{def: discrete norms}
For all $\vh{u}{} \in \bbf{X}_h$ and $p_h    \in M_h$ we define
    \begin{align}
    &\tnorme{(\vh{u}{},\ph{p}{})} = \tnormx{\vh{u}{}}{} + \tnormp{\ph{p}{}}{},\\
    &\tnormx{\vh{u}{}} = \normx{\vh{u}{S}}{S} + \normx{\vh{u}{D}}{D} + \left(\sum_{F\in\Gamma} \sigma_\Gamma \| \jumpl \vh{u}{} \jumpr_\bbf{n} \|^2_{L^2(F)} \right)^\frac{1}{2},\\
    &\tnormp{\ph{p}{}}{} = \normp{\ph{p}{S}}{S} + \normp{\ph{p}{D}}{D} + \left(\sum_{F\in\Gamma} \xi_\Gamma \| \jumpl \ph{p}{} \jumpr_\bbf{n} \|^2_{L^2(F)} \right)^\frac{1}{2},
\end{align}
where
\begin{align}
    &\normx{\vh{u}{S}}{S} = \|\vh{u}{S}\|_{dG} + \left(\sum_{j=1}^{d-1}\sum_{F \in \Gamma} \rho^{-1} \|\vh{u}{S} \cdot \bbf{t}_{\Gamma,j} \|_{L^2(F)}^2 \right)^\frac{1}{2}, \\
    &\normx{\vh{u}{D}}{D} = \|\vh{u}{D}\|_{L^2(\Omega_D)} + \left( \sum_{F\in\scr{F}^i_\mcal{D}} \sigma_D \| \jumpl \vh{u}{D} \jumpr_\bbf{n} \|_{L^2(F)}^2 \right)^\frac{1}{2}, \\
    &\|\vh{u}{I}\|_{dG} = \|\nabla_h \vh{u}{I} \|_{L^2(\Omega_I)} + \left(\sum_{F\in\scr{F}_\mcal{I}} \sigma_\mcal{I} \|\jumpl\vh{u}{I}\jumpr \|_{L^2(F)}^2 \right)^\frac{1}{2}, \\
    &\normp{\ph{p}{I}}{I} = \|\ph{p}{I}\|_{L^2(\Omega_I)} + \left(\sum_{F\in\scr{F}^i_\mcal{I}} \xi_\mcal{I} \| \jumpl \ph{p}{I} \jumpr_\bbf{n} \|_{L^2(F)}^2 \right)^\frac{1}{2}, \qquad \text{ with } \mcal{I} = \{ S,D \},
\end{align}
having denoted with $\nabla_h$ the piecewise broken gradient operator. 
Additionally, we define:
\begin{align}
    &|\!|\!| \vh{u}{} |\!|\!|_{dG} = \| \vh{u}{S} \|_{dG} + \| \vh{u}{D} \|_{dG} + \left(\sum_{F\in\Gamma} \sigma_\Gamma \| \jumpl \vh{u}{} \jumpr \|^2_{L^2(F)} \right)^\frac{1}{2}, \\
    &\tnormdiv{\vh{u}{}}{} = \tnormx{\vh{u}{}}{} + \normll{\nabla_h \cdot \vh{u}{D}}{D}.   
\end{align}
\end{definition}

\subsection{Properties of the coupling bilinear form}
Here, we prove that $\bh{\cdot}{\cdot}{}$ is inf-sup stable and continuous. In particular, we extend the inf-sup condition for the Stokes problem on polygonal grids \cite{AntoniettiMascottoVeraniZonca2022} to the non-Newtonian Stokes--Darcy coupled system \eqref{eq: discrete SDF}.\\

We now prove the following generalized inf-sup condition:
\begin{theorem}\label{thm: discrete infsup}
    For all $\ph{p}{} \in M_h$, it holds
    \begin{equation}\label{eq:discrete infsup}
    \sup_{\vh{v}{}\in \bbf{X}_h, \vh{v}{} \neq \bbf{0}} \frac{\bh{\vh{v}{}}{\ph{p}{}}{}}{\tnormdiv{\vh{v}{}}{}} + |\ph{p}{}|_{J} \gtrsim \|\ph{p}{}\|_{L^2(\Omega)},
\end{equation}
where $|\ph{p}{}|_{J}^2 = s_h(\ph{p}{}, \ph{p}{})$.
\end{theorem}
\begin{proof}
This proof is divided into two parts. First, we first establish the discrete inf-sup condition for the coupled problem with respect to the $|\!|\!| \cdot |\!|\!|_{dG}$ norm of the velocity $\vh{v}{}$ \cite{AntoniettiMascottoVeraniZonca2022} under a weighted norm on the facets comprising the interface $\Gamma$. To this end, we consider the limiting case in which the entire fluid domain is governed by the Stokes equations, and demonstrate that the inf-sup condition holds for a general weighted norm. Then we show how the inf-sup condition for the Stokes problem can be extended to the coupled case, using arguments involving inequalities among norms.
We now proceed by exposing the first part of the proof, introducing the following weighted average operator:
\begin{equation}
    \averagel \phi \averager_{\omega} = \omega \phi_1 + (1-\omega) \phi_2,
\end{equation}
and we take
\begin{equation} \label{omega equation}
    \omega = \begin{cases}
        \frac{1}{2} & \text{ on } \scr{F}_\mcal{S} \cup \scr{F}_\mcal{D} \\
        0 & \text{ on } \Gamma
    \end{cases}.
\end{equation}
With this definition of average, we can still apply the following equality, c.f. \cite{BrezziMarini}:
\begin{equation}\label{eq: weigh magic formula}
    \jumpl \phi \psi \jumpr = \jumpl \phi \jumpr \averagel \psi \averager_{\omega} + \jumpl \psi \jumpr \averagel \phi \averager_{\bar{\omega}},
\end{equation} 
where $\bar{\omega} = 1 - \omega$.
Indeed, 
\begin{equation}
\begin{split}
    \jumpl \phi \jumpr \averagel \psi \averager_{\omega} + \jumpl \psi \jumpr \averagel \phi \averager_{\bar{\omega}} &= (\phi_1 - \phi_2)(\omega\psi_1 + (1-\omega)\psi_2) + (\psi_1 - \psi_2)((1-\omega)\phi_1 + \omega \phi_2) \\ &= \psi_1\phi_1 - \psi_2\phi_2 = \jumpl \phi \psi \jumpr. 
\end{split}
\end{equation}
Reasoning as in \cite{AntoniettiMascottoVeraniZonca2022}, we recall that at the continuous level for any $\ph{q}{} \in M_h \subset L^2_0(\Omega)$ there exists $\bbf{v}^{q_h} \in\bbf{X}$ such that:
\begin{equation}
    \nabla \cdot \bbf{v}^{q_h} = \ph{q}{} \ \ \ \| \bbf{v}^{q_h} \|_\bbf{X} \lesssim \|\ph{q}{}\|_{L^2(\Omega)}.
\end{equation}
Applying an element-wise integration by part, using $\jumpl \bbf{v}^{q_h}\jumpr = 0$ for any $F \in \scr{F}_h$, observing that $\nabla \ph{q}{} \in \bbf{X}
_h$ and considering the global polynomial interpolation operator $\bbf{\Pi}^l:\bbf{X}\to \bbf{X}_h$ of \cite{AntoniettiMascottoVeraniZonca2022} we get
\begin{equation}
    \begin{split}
        \|\ph{q}{}\|^2_{L^2(\Omega)} &= \int_{\Omega} \ph{q}{}\nabla\cdot\bbf{v}^{q_h} = -\int_\Omega \nabla \ph{q}{} \cdot\bbf{v}^{q_h} + \sum_{F\in\scr{F}_h}\int_F \jumpl \ph{q}{}\jumpr \averagel\bbf{v}^{q_h}\averager \\ & = - \int_{\Omega} \nabla \ph{q}{} \cdot\bbf{\Pi}^l\bbf{v}^{q_h} + \int_{\Omega} \nabla \ph{q}{} \cdot(\bbf{\Pi}^l\bbf{v}^{q_h} - \bbf{v}^{q_h}) + \sum_{F\in\scr{F}_h}\int_F \jumpl \ph{q}{}\jumpr \averagel\bbf{v}^{q_h}\averager_{\omega} \\ & = \int_\Omega \ph{q}{}\nabla\cdot\bbf{\Pi}^l\bbf{v}^{q_h} - \sum_{F \in \scr{F}_h}\int_F \jumpl \ph{q}{} \jumpr \averagel \bbf{\Pi}^l\bbf{v}^{q_h} \averager_{\omega} - \sum_{F \in \scr{F}_h}\int_F \averagel \ph{q}{} \averager_{\bar{\omega}} \jumpl \bbf{\Pi}^l\bbf{v}^{q_h} \jumpr \\ & + \sum_{F\in\scr{F}_h}\int_F \jumpl \ph{q}{}\jumpr \averagel\bbf{v}^{q_h}\averager_{\omega} + \int_{\Omega} \nabla \ph{q}{} \cdot(\bbf{\Pi}^l\bbf{v}^{q_h} - \bbf{v}^{q_h}) \\ & = - b^h(\bbf{\Pi}^l\bbf{v}^{q_h}, \ph{q}{}) + \int_{\Omega} \nabla \ph{q}{} \cdot(\bbf{\Pi}^l\bbf{v}^{q_h} - \bbf{v}^{q_h}) + \sum_{F\in\scr{F}_h}\int_F \jumpl \ph{q}{}\jumpr \averagel(\bbf{v}^{q_h} - \bbf{\Pi}^l\bbf{v}^{q_h})\averager_{\omega}.
    \end{split}
\end{equation}
Now the proof follows exactly the steps reported on \cite{AntoniettiMascottoVeraniZonca2022} to obtain
\begin{equation}\label{eq: stokes discrete inf sup}
\sup_{\vh{v}{}\in \bbf{X}_h, \vh{v}{} \neq \bbf{0}}\frac{\bh{\vh{v}{}}{\ph{p}{}}{}}{|\!|\!|\vh{v}{}|\!|\!|_{dG}} + |\ph{p}{}|_{J} \gtrsim \|\ph{p}{}\|_{L^2(\Omega)}.
\end{equation}
Now, we proceed with the second part of the proof starting from \eqref{eq: stokes discrete inf sup}, and recall that
\begin{equation}
\|\jumpl\vh{v}{D}\jumpr\|_{L^2(F)} \geq \|\jumpl\vh{v}{D}\jumpr_\bbf{n}\|_{L^2(F)},
\qquad\text{ and }\qquad
\| \nabla_h \vh{v}{} \|_{L^2(\Omega)} \geq \| \nabla_h \cdot \vh{v}{} \|_{L^2(\Omega)}.
\end{equation}
Additionally, by a discrete Poincarè inequality (cf. \cite[Corollary 5.4]{SobolevEmbedding} and \cite[Theorem 1.8]{Botti.Mascotto:26}), we have
\begin{equation}
     \|\vh{v}{D}\|_{dG} \gtrsim \|\vh{v}{D}\|_{L^2(\Omega_D)},
\end{equation}
and, as a result, we infer
\begin{equation}\label{eq: normX1h maggiore normX2h}
    \|\vh{v}{D}\|_{dG} \gtrsim \normx{\vh{v}{D}}{D} + \normll{\nabla_h \cdot \vh{v}{D}}{D}.
\end{equation}
Next, we apply the global discrete trace inequality $\| \vh{v}{S} \|_{L^2(\partial \Omega_S)} \lesssim \|\vh{v}{S}\|_{dG}$, which is a consequence of \cite[Theorem 6.7]{HighOrderpolynomialapprox}, to write
\begin{equation}
    \sum_{j=1}^{d-1} \sum_{F \in \Gamma } \rho^{-1} \| \vh{v}{S} \cdot \bbf{t}_{\Gamma,j} \|_{L^2(F)}^2\lesssim \rho^{-1} \| \vh{v}{S} \|_{L^2(\Gamma)}^2 
    \lesssim \rho^{-1} \| \vh{v}{S} \|_{L^2(\partial\Omega_S)}^2
    \lesssim \rho^{-1}\|\vh{v}{S}\|_{dG}^2. 
\end{equation}
As a result, owing to the definition of the
$\normx{\cdot}{S}$-norm, we have
\begin{equation} \label{eq: normatr1 minora nomraX1h}
    \|\vh{v}{S}\|_{dG} \gtrsim \normx{\vh{v}{S}}{S}.
\end{equation}
Hence, applying \eqref{eq: normX1h maggiore normX2h} and \eqref{eq: normatr1 minora nomraX1h} we obtain:
\begin{equation}
    |\!|\!|\vh{v}{}|\!|\!|_{dG} \gtrsim \tnormdiv{\vh{v}{}}{} \Rightarrow \frac{1}{\tnormdiv{\vh{v}{}}{}} \gtrsim \frac{1}{|\!|\!|\vh{v}{} |\!|\!|_{dG}}.
\end{equation}
Finally, from \eqref{eq: stokes discrete inf sup}, we can conclude that:
\begin{equation}\label{discrete inf sup}
\sup_{\vh{v}{}\in \bbf{X}_h, \vh{v}{}\neq 0}\frac{\bh{\vh{v}{}}{\ph{p}{}}{}}{\tnormdiv{\vh{v}{}}{}} + |\ph{p}{}|_J \gtrsim \sup_{\vh{v}{}\in \bbf{X}_h, \vh{v}{} \neq 0}\frac{\bh{\vh{v}{}}{\ph{p}{}}{}}{|\!|\!| \vh{v}{} |\!|\!|_{dG}} + |\ph{p}{}|_J \gtrsim \|\ph{p}{}\|_{L^2(\Omega)}.
\end{equation}
\end{proof}

We next state the following results.
\begin{corollary}\label{lem: bh inf-sup E}
For all $\ph{p}{} \in M_h$ it holds
    \begin{equation}\label{eq:discrete infsup E}
    \sup_{\vh{v}{}\in X_h, \vh{v}{} \neq 0} \frac{\bh{\vh{v}{}}{\ph{p}{}}{}}{\tnormx{\vh{v}{}}{}} + |\ph{p}{}|_{J} \gtrsim \|\ph{p}{}\|_{L^2(\Omega)},
\end{equation}
where $|\ph{p}{}|_{J}^2 = s_h(\ph{p}{}, \ph{p}{})$. 
\end{corollary}
\begin{proof}
    Starting from Theorem \ref{thm: discrete infsup}, we observe that, by definition: $ \tnormdiv{\vh{v}{}}{} \gtrsim \tnormx{\vh{v}{}}{},$    which implies that:
    \begin{equation}
        \frac{1}{\tnormx{\vh{v}{}}{}} \gtrsim \frac{1}{\tnormdiv{\vh{v}{}}{}}.
    \end{equation}
Therefore,
    \begin{equation}
         \sup_{\vh{v}{}\in \bbf{X}_h, \vh{v}{} \neq 0} \frac{\bh{\vh{v}{}}{\ph{p}{}}{}}{\tnormx{\vh{v}{}}{}} + |\ph{p}{}|_{J} \gtrsim \sup_{\vh{v}{}\in \bbf{X}_h, \vh{v}{} \neq 0} \frac{\bh{\vh{v}{}}{\ph{p}{}}{}}{\tnormdiv{\vh{v}{}}{}} + |\ph{p}{}|_{J} \gtrsim \|\ph{p}{}\|_{L^2(\Omega)},
    \end{equation}
    and the proof is complete.
\end{proof}
The next lemma establishes the existence of a generalized right-inverse of the discrete divergence operator.
\begin{lemma} \label{lem: discrete right-inverse}
For each $p_h \in M_h$, there exists $\bm{\xi}_h \in \bbf{X}_h$ such that 
\begin{equation}\label{eq: from inf-sup}
    \|\ph{p}{}\|_{L^2(\Omega)}^2 \lesssim
    b_h(\bm{\xi}_h, \ph{p}{}) + \sh{\ph{p}{}}{\ph{p}{}} \qquad\text{and}\qquad 
    \tnormx{\bm{\xi}_h}{} \le \|\ph{p}{}\|_{L^2(\Omega)}.
\end{equation}
\end{lemma}
\begin{proof}
    The discrete inf-sup inequality of Corollary \ref{lem: bh inf-sup E} implies, for all $p_h\in M_h$, the existence of $\vh{\overline{v}}{}\in \bbf{X}_h$ such that \begin{equation} \label{eq: inf-sup for v_h bar}
        \frac{\bh{\vh{\overline{v}}{}}{\ph{p}{}}{}}{\tnormx{\vh{\overline{v}}{}}{}} + |\ph{p}{}|_J \gtrsim \| \ph{p}{} \|_{L^2(\Omega)}.
    \end{equation}
    Then, we take $\bm{\xi}_h =  \frac{\| \ph{p}{} \|_{L^2(\Omega)}\vh{\overline{v}}{}}{\tnormx{\vh{\overline{v}}{}}{}}$
    and observe that, by the definition of $\bm{\xi}_h$, we have $\tnormx{\bm{\xi}_h}{} \le \| \ph{p}{} \|_{L^2(\Omega)}$. 
    To prove the first inequality in \eqref{eq: from inf-sup}, we observe that
    \begin{equation}
        \bh{\bm{\xi}_h}{\ph{p}{}}{} + |\ph{p}{}|_J^2 = \left(\frac{\bh{\vh{\overline{v}}{}}{\ph{p}{}}{}}{\tnormx{\vh{\overline{v}}{}}{}} + |\ph{p}{}|_J\right)\|\ph{p}{}\|_{L^2(\Omega)} - |\ph{p}{}|_J \|\ph{p}{}\|_{L^2(\Omega)} + |\ph{p}{}|_J^2.
    \end{equation}
    Therefore, applying \eqref{eq: inf-sup for v_h bar} and the Young inequality we obtain
    \begin{equation}
        \begin{split}
            \bh{\bm{\xi}_h}{\ph{p}{}}{} + |\ph{p}{}|_J^2 
            &\gtrsim \| \ph{p}{} \|^2_{L^2(\Omega)} - |\ph{p}{}|_J \|\ph{p}{}\|_{L^2(\Omega)} + |\ph{p}{}|_J^2 \\ & \gtrsim\frac{\|\ph{p}{}\|_{L^2(\Omega)}^2 + |\ph{p}{}|_J^2}{2} \gtrsim \|\ph{p}{}\|_{L^2(\Omega)}^2.
        \end{split}
    \end{equation}
\end{proof}
The continuity of the bilinear form $b_h(\cdot, \cdot)$ is established in the following lemma; its proof is reported in Appendix~\ref{app: proof bh cont}.
\begin{lemma}\label{lem: bh cont}
    The bilinear form $\bh{\vh{v}{}}{\ph{p}{}}{} = \bh{\vh{v}{S}}{\ph{p}{S}}{S} + \bh{\vh{v}{D}}{\ph{p}{D}}{D} + \bh{\vh{v}{}}{\ph{p}{}}{\Gamma}$ is continuous, i.e.
    \begin{equation}
        |\bh{\vh{v}{}}{\ph{p}{}}{}| \lesssim \tnormdiv{\vh{v}{}}{} \tnormp{\ph{p}{}}{} \qquad \forall \vh{v}{}\in\bbf{X}_h, \forall \ph{p}{}\in M_h.
    \end{equation}
\end{lemma}

\subsection{Well-posedness of the discrete problem}
We start by proving the continuity and monotonicity of $\ah{\cdot}{\cdot}{}$ in the following lemmas, whose proofs can be found in Appendix~\ref{app: proof ah cont}.
\begin{lemma}\label{lem: ah cont}
    \pa{Let Assumption~\ref{assu: condition on g1 and g2} and Assumption~\ref{assump: mesh properties} be satisfied.} Then,
    \begin{equation}\label{eq: Ah bound}
        |a_h(\vh{u}{}, \vh{w}{})-\ah{\vh{v}{}}{\vh{w}{}}{}| \lesssim \tnormx{\vh{u}{}-\vh{v}{}}{}\tnormx{\vh{w}{}}{} \qquad \forall \vh{u}{}, \vh{v}{}, \vh{w}{} \in \bbf{X}_h.
    \end{equation}
\end{lemma}
\begin{lemma}\label{lem: ah mon}
    \pa{Let Assumption~\ref{assu: condition on g1 and g2} and Assumption~\ref{assump: mesh properties} be satisfied.} Then for sufficiently large $\sigma_S$, $a_h(\cdot,\cdot)$ is monotone, i.e.,
\begin{equation}\label{eq: Ah monotonicity}
    (a_h(\vh{u}{}, \vh{u}{} - \vh{v}{}) - a_h(\vh{v}{}, \vh{u}{} - \vh{v}{})) \gtrsim \tnormxq{\vh{u}{} - \vh{v}{}}{} \qquad \forall \vh{u}{}, \vh{v}{} \in \bbf{X}_h.
\end{equation}
\end{lemma}
We rewrite problem \eqref{eq: discrete SDF} in the form
\begin{equation}\label{eq: probuni}
    \mcal{A}_h((\vh{u}{},\ph{p}{}),(\vh{v}{},\ph{q}{})) = F((\vh{v}{}, \ph{q}{})),
\end{equation}
where 
\begin{equation}\label{def: Ah}
    \mcal{A}_h((\vh{u}{},\ph{p}{}),(\vh{v}{},\ph{q}{})) = \ah{\vh{u}{}}{\vh{v}{}}{} + \bh{\vh{v}{}}{\ph{p}{}}{} - \bh{\vh{u}{}}{\ph{q}{}}{} + \sh{\ph{p}{}}{\ph{q}{}}.
\end{equation}
To prove the well-posedness of \eqref{eq: probuni}, we refer to \cite{MappedCoercivity} (Theorem 4), and we need to prove that the form $\mcal{A}_h$ is linearly mapped coercive, in the sense of \cite{MappedCoercivity}.
\begin{lemma}
Let Assumption~\ref{assu: condition on g1 and g2} and Assumption~\ref{assump: mesh properties} be satisfied.  The form $\mcal{A}_h(\cdot,\cdot)$ defined as \eqref{def: Ah} is linearly mapped coercive.
\end{lemma}
\begin{proof}
    We need to prove that for any $(\vh{u}{}, \ph{p}{}) \in \bbf{X}_h \times M_h$
    \begin{equation}
        \mcal{A}_h((\vh{u}{},\ph{p}{}),\Phi(\vh{u}{},\ph{p}{}))\gtrsim \tnorme{(\vh{u}{},\ph{p}{})}{}^2,
    \end{equation}
    with $\Phi$ a bijection. We recall that the inf-sup condition \eqref{eq:discrete infsup E} implies Lemma~\ref{lem: discrete right-inverse}.
    Now, we define $\Phi$ as:
    \begin{equation}
        \Phi(\vh{u}{},\ph{p}{}) = (\alpha\vh{u}{}+\beta\vh{\bm{\xi}}{}, \alpha\ph{p}{}),
    \end{equation}
    with $\alpha$ and $\beta$ to be appropriately chosen.
    By definition $\Phi(\vh{u}{},\ph{p}{})$ is linear. To prove that is a bijection we simply need to prove that is an injection, i.e. $\Phi(\vh{u}{},\ph{p}{}) = (\bm{0},0) \Leftrightarrow (\vh{u}{},\ph{p}{}) = (\bm{0},0)$.
    On the one hand if $\Phi(\vh{u}{},\ph{p}{}) = (\bbf{0}, 0)$ this implies, 
    \begin{align}
        & \alpha \vh{u}{} + \beta \vh{\bm{\xi}}{} = \bm{0}, \\
        & \alpha \ph{p}{} = 0,
    \end{align}
    From this we derive that $\alpha \ph{p}{} = 0 \Rightarrow \ph{p}{} = 0.$
    Condition \eqref{eq: from inf-sup} implies that also $\vh{\bm{\xi}}{} = 0$. From this we have $\alpha\vh{u}{}+\beta\vh{\bm{\xi}}{} = \bm{0} \Rightarrow \vh{u}{} = \bm{0}.$
    The inverse implication is trivial.
    Now, we write:
    \begin{equation}\label{eq; Ah inequality}
        \begin{split}
            \mcal{A}_h((\vh{u}{},\ph{p}{}),\Phi(\vh{u}{},\ph{p}{})&) = \mcal{A}_h((\vh{u}{},\ph{p}{}),(\alpha\vh{u}{}+\beta\vh{\bm{\xi}}{},\alpha\ph{p}{})) \\& = \ah{\vh{u}{}}{\alpha\vh{u}{}+\beta\vh{\bm{\xi}}{}}{} + \bh{\alpha\vh{u}{}+\beta\vh{\bm{\xi}}{}}{\ph{p}{}}{} - \bh{\vh{u}{}}{\alpha\ph{p}{}}{} + \sh{\ph{p}{}}{\alpha\ph{p}{}} \\& = \alpha\ah{\vh{u}{}}{\vh{u}{}}{} + \beta\ah{\vh{u}{}}{\vh{\bm{\xi}}{}}{} + \alpha\bh{\vh{u}{}}{\ph{p}{}}{} \\ &+ \beta\bh{\vh{\bm{\xi}}{}}{\ph{p}{}}{} - \alpha\bh{\vh{u}{}}{\ph{p}{}}{} + \alpha\sh{\ph{p}{}}{\ph{p}{}} \\ & = \alpha\ah{\vh{u}{}}{\vh{u}{}}{} + \beta\ah{\vh{u}{}}{\vh{\bm{\xi}}{}}{} + \beta\bh{\vh{\bm{\xi}}{}}{\ph{p}{}}{} + \alpha\sh{\ph{p}{}}{\ph{p}{}}.
        \end{split}
    \end{equation}
    We study the different terms separately. From Lemma~\ref{lem: ah mon} we recover:
    \begin{equation}
        \alpha \ah{\vh{u}{}}{\vh{u}{}}{} \gtrsim \alpha \tnormx{\vh{u}{}}{}^2.
    \end{equation}
    From Lemma~\ref{lem: ah cont} and Young's inequality we recover:
    \begin{equation}
    \beta(\ah{\vh{u}{}}{\bm{\xi}_h}{}) \gtrsim - \beta \tnormx{\vh{u}{}}{}\tnormx{\vh{\xi}{}}{} \gtrsim - \beta \left( \frac{\tnormx{\vh{u}{}}{}^2}{2\epsilon} + \epsilon\frac{\tnormx{\vh{\xi}{}}{}^2}{2}\right),
\end{equation}
with Young's constant $\epsilon > 0$ to be properly chosen. Furthermore, using the inf-sup condition \eqref{eq: from inf-sup} we get:
\begin{equation}
    \beta\bh{\bm{\xi}_h}{\ph{p}{}}{} \geq \beta \|\ph{p}{}\|^2_{L^2(\Omega)} - \beta |\ph{p}{}|^2_J.
\end{equation}
Considering the last inequalities, choosing $\epsilon \simeq \frac{\beta}{\alpha}$, and considering \eqref{eq: from inf-sup}, equation \eqref{eq; Ah inequality} becomes:
\begin{equation}
    \mcal{A}_h((\vh{u}{},\ph{p}{}),\Phi(\vh{u}{},\ph{p}{})) \gtrsim \frac{\alpha}{2}\tnormx{\vh{u}{}}{}^2 + \left[ - \frac{(\beta)^2}{2\alpha} + \beta \right]\|\ph{p}{}\|^2_{L^2(\Omega)} + (\alpha - \beta) |\ph{p}{}|_J^2.
\end{equation}
Choosing $\alpha \geq \beta$ we obtain:
\begin{equation}
    \mcal{A}_h((\vh{u}{},\ph{p}{}),\Phi(\vh{u}{},\ph{p}{})) \gtrsim \beta \| (\vh{u}{}, \ph{p}{}) \|_{E}^2,
\end{equation}
and the proof is complete.
\end{proof}

\subsection{A-priori stability estimate} \label{sec: stability}
In this section, we derive an a-priori stability estimate for the discrete solution $(\vh{u}{}, \ph{p}{}) \in \bbf{X}_h \times M_h$ of problem \eqref{eq: discrete SDF}.
\begin{proposition}\label{prop: stability}
    Let Assumption~\ref{assu: condition on g1 and g2} and Assumption~\ref{assump: mesh properties} be satisfied. Moreover, assume that that the penalty parameters $\sigma_S, \sigma_D$ and $\sigma_\Gamma$ are large enough. \pa{Let $(\vh{u}{}, \ph{p}{})\in \bbf{X}_h\times M_h$ be the unique discrete solution of \eqref{eq: discrete SDF}.} Then,
    \begin{equation}
       \tnormxq{\vh{u}{}}{} + \| \nabla_h \cdot \vh{u}{D} \|^2_{L^2(\Omega_D)} + \tnormpq{\ph{p}{}}{} \lesssim \|\bbf{f}_S\|_{L^2(\Omega_S)}^2 + \|\bbf{f}_D\|_{L^2(\Omega_D)}^2.
    \end{equation}
\end{proposition}
\begin{proof}
We consider the discrete problem \eqref{eq: discrete SDF} and we sum the two equations:
\begin{equation}
    \ah{\vh{u}{}}{\vh{v}{}}{} + \bh{\vh{v}{}}{\ph{p}{}}{} - \bh{\vh{u}{}}{\ph{q}{}}{} + \sh{\ph{p}{}}{\ph{q}{}} = F(\vh{v}{}).
\end{equation}
Taking $\vh{v}{} = \vh{u}{}$ and $\ph{q}{} = \ph{p}{}$ we obtain:
\begin{equation}
    \begin{split}
        \ah{\vh{u}{}}{\vh{u}{}}{} + \sh{\ph{p}{}}{\ph{q}{}} &= F(\vh{u}{}) \leq \| \bbf{f}_S \|_{L^2(\Omega_S)} \| \vh{u}{S} \|_{L^2(\Omega_S)} + \| \bbf{f}_D \|_{L^2(\Omega_D)} \| \vh{u}{D} \|_{L^2(\Omega_D)} \\ & \lesssim \| \bbf{f}_S \|_{L^2(\Omega_S)}^2 + \| \vh{u}{S} \|_{L^2(\Omega_S)}^2 + \| \bbf{f}_D \|_{L^2(\Omega_D)}^2 + \| \vh{u}{D} \|_{L^2(\Omega_D)}^2.
    \end{split}
\end{equation}
Finally, we get:
\begin{equation}\label{eq: stab1}
    \tnormxq{\vh{u}{}}{} + |\ph{p}{}|_J^2 \lesssim  \| \bbf{f}_S \|_{L^2(\Omega_S)}^2 + \| \bbf{f}_D \|_{L^2(\Omega_D)}^2.
\end{equation}
Now we study the bound of $\|\nabla_h \cdot \vh{u}{D} \|_{L^2(\Omega_D)}$. To do so we consider the second equation of \eqref{eq: discrete SDF} and test it against $\ph{\Tilde{q}}{} = (0, \nabla \cdot \vh{u}{D})$, to obtain
\begin{equation}
    -\bh{\vh{u}{}}{\ph{\Tilde{q}}{}}{} + \sh{\ph{p}{}}{\ph{\Tilde{q}}{}} = 0.
\end{equation}
From the previous identity we can derive the following:
\begin{equation}\label{eq: bound norm div}
    \begin{split}
        \|\nabla \cdot \vh{u}{D}\|^2_{L^2(\Omega_D)} & \lesssim \sum_{F \in \scr{F}_\mcal{D}} \sigma_D \| \jumpl\vh{u}{D} \jumpr_\bbf{n} \|_{L^2(F)}^2 + \sum_{F \in \Gamma} \sigma_\Gamma \| \jumpl\vh{u}{} \jumpr_\bbf{n} \|_{L^2(F)}^2 \\ & + \sum_{F \in \scr{F}_\mcal{D}^i} \xi_D \| \jumpl \ph{p}{D} \jumpr_\bbf{n} \|^2_{L^2(F)} + \sum_{F \in \Gamma} \xi_\Gamma \| \jumpl \ph{p}{} \jumpr_\bbf{n} \|^2_{L^2(F)}.
    \end{split}
\end{equation}
The complete derivation of the estimate can be found in Appendix~\ref{app: bound div}.
Now, we consider the discrete inf-sup condition \eqref{eq:discrete infsup E} and we observe that:
\begin{equation}
    \bh{\vh{v}{}}{\ph{p}{}}{} = F(\vh{v}{}) - \ah{\vh{u}{}}{\vh{v}{}}{} \lesssim \|\bbf{f}_\mcal{S}\|_{L^2(\Omega_S)}\normx{\vh{v}{S}}{S}+ \|\bbf{f}_\mcal{D}\|_{L^2(\Omega_D)}\normx{\vh{v}{D}}{D}+ \tnormx{\vh{u}{}}{} \tnormx{\vh{v}{}}{}.
\end{equation}
Then, we recover:
\begin{equation}\label{eq: stab2}
    \|\ph{p}{}\|_{L^2(\Omega)} \lesssim \sup_{\vh{v}{}\in X_h} \frac{\|\bbf{f}_\mcal{S}\|_{L^2(\Omega_S)}\normx{\vh{v}{S}}{S}+ \|\bbf{f}_\mcal{D}\|_{L^2(\Omega_D)}\normx{\vh{v}{D}}{D}+ \tnormx{\vh{u}{}}{} \tnormx{\vh{v}{}}{}}{\tnormx{\vh{v}{}}{}} + |\ph{p}{}|_J.
\end{equation}
Now we observe that summing up Equations \eqref{eq: stab1}, \eqref{eq: stab2} and \eqref{eq: bound norm div}, we obtain the following estimate:
\begin{equation}\label{eq: big stability}
\tnormxq{\vh{u}{}}{} + \| \nabla_h \cdot \vh{u}{D} \|^2_{L^2(\Omega_D)} + \tnormpq{\ph{p}{}}{} \lesssim \|\bbf{f}_S\|_{L^2(\Omega_S)}^2 + \|\bbf{f}_D\|_{L^2(\Omega_D)}^2,
\end{equation}
and the proof is complete.
\end{proof}

\section{Error estimates}\label{sec: error estimates}
In this section we prove a bound for $\normx{\bbf{u} - \vh{u}{}}{}$ and $\normp{p - \ph{p}{}}{}$. By applying the triangular inequality, we only need to find a bound for $\normx{\vh{u}{} - \bbf{u}_I}{}$ and $\normp{\ph{p}{} - p_I }{}$ where $\bbf{u}_I$ and $p_I$ are suitable interpolation of the continuous solution. \\

First, we introduce the following definition for the element-wise $L^2$-orthogonal projection:
$\pi_h^{0,l}: L^1(\Omega) \to P^l(\scr{K})$ defined such that: for all $v \in L^1(\Omega)$ and all $K \in \scr{K}$
    \begin{equation}
        (\pi_h^{0,l}v)_{|_K} = \pi_K^{0,l}v,
        \qquad\text{ with }\quad
        \int_K (\pi_K^{0,l}v) \;w_h = 
        \int_K v \; w_h, \quad
        \forall w_h\in P^l(K).
    \end{equation}
We also recall the following result from \cite{HighOrderpolynomialapprox}:
\begin{theorem}[Approximation properties of the $L^2$-orthogonal projector]
\label{thm: Approximation property}
\pa{Let Assumption~\ref{assump: mesh properties} be satisfied.}
Let a polynomial degree $l \geq 0$, an integer $s \in \{0, ..., l+1 \}$ and a real number $p \in [1, \infty ]$ be given. Then, for any $K$ element or face of $\scr{M}$, all $v \in W^{s,p}(K)$, and all $m \in \{0, ..., s \}$,
    \begin{equation}\label{eq: Approx estimate}
        |v - \pi^{0,l}_{K}v|_{W^{m,p(K)}} \lesssim h_{K}^{s-m}|v|_{W^{s,p}(K)}.
    \end{equation}
    Moreover, if $s \geq 1$, for all $K \in \scr{K}$, all $v \in W^{s,p}(K)$, all $F \in \scr{F}_{K}$, and all $m \in \{ 0, ..., s-1 \}$, it holds that
    \begin{equation} \label{eq: approx estimate boundary}
        h_K^\frac{1}{p} |v - \pi^{0,l}_K v|_{W^{m,p}(F)} \lesssim h_K^{s - m} |v|_{W^{s,p}(K)}.
    \end{equation}
\end{theorem}
To simplify the notation, we define the interpolation errors and the discrete errors: $\bbf{e}_I^u = \bbf{u}_I - \bbf{u}$ and $e^p_I = p_I - p$ are the velocity and pressure interpolation errors, and $\evh{e}{u} = \vh{u}{} - \bbf{u}_I$ and $\eph{e}{p} = \ph{p}{} - p_I$  are the discrete errors, where $(\bbf{u}, p)$ is the solution of the continuous problem and $(\vh{u}{}, \ph{p}{})$ is the solution of the discrete problem. Additionally, we define: $\bbf{e}^u = \vh{u}{} - \bbf{u}$ and $e^p = \ph{p}{} - p$, and we observe that $\bbf{e}^u = \evh{e}{u} + \bbf{e}^u_I$ and $e^p = \eph{e}{p}+e^p_I$.

Now, we state the following error estimate theorem.

\begin{theorem}\label{thm: error estimate}
    \pa{Let Assumption~\ref{assu: condition on g1 and g2} and Assumption~\ref{assump: mesh properties} be satisfied.} Let the mesh size be such that $h \lesssim \rho$ and assume that the penalty parameters \eqref{eq: penalty} are large enough. Additionally, we assume that the solutions to \eqref{eq: weak SDF} satisfy
    $$\bbf{u}_\mcal{S} \in [H^{t_S + 1}(\Omega_\mcal{S})]^d,\; \bbf{u}_\mcal{D} \in [H^{t_D + 1}(\Omega_\mcal{D})]^d,\; p_\mcal{S} \in H^{m_S + 1}(\Omega_\mcal{S})\, \text{ and }\, p_\mcal{D} \in H^{m_D + 1}(\Omega_\mcal{D}).$$
    Then, it holds:
    \begin{equation}
    \begin{split}
        \tnormeq{(\evh{e}{u},\eph{e}{p})}{} \lesssim& h^{2t_S}|\bbf{u}^S|_{H^{t_S+1}(\Omega_S)}^2 + h^{2t_D}|\bbf{u}^D|_{H^{t_D+1}(\Omega_D)}^2 + h^{2t}|\bbf{u}|_{H^{t+1}(\Omega)}^2 \\ &+ h^{2m_S+2} |p^\mcal{S}|_{H^{m_S+1}(\Omega_\mcal{S})} + h^{2m_D+2} |p^\mcal{D}|_{H^{m_D+1}(\Omega_\mcal{D})},
    \end{split}
\end{equation}
where $t_S$ and $t_D$ are the polynomial degrees for the Stokes' and Darcy's velocity fields,  $m_S$ and $m_D$ are the polynomial degrees of the Stokes and Darcy's pressures, resepctively,  and
$t = \min\{t_\mcal{S}, t_\mcal{D}\}$.
\end{theorem}
\begin{rmk}
\pa{We observe that the error bounds are consistent with the error convergence theory for the Discontinuous Galerkin method applied to the Stokes equation \cite{RiviereDG}. }
\end{rmk}
\begin{proof}
We observe that the following consistency property holds
\begin{equation}\label{eq: consistency (u,p)}
\begin{cases}
\begin{aligned}
\ah{\bbf{u}}{\vh{v}{}}{} + \bh{\vh{v}{}}{p}{} 
&= (\bbf{f}, \vh{v}{}) 
&& \forall \vh{v}{} \in \bbf{X}_h, \\
-\bh{\bbf{u}}{\ph{q}{}}{}+\sh{p}{\ph{q}{}} 
&= 0 
&& \forall \ph{q}{} \in M_h.
\end{aligned}
\end{cases}
\end{equation}
assuming that the continuous solutions $(\bbf{u}, p)$ are smooth enough.
Now, we subtract \eqref{eq: consistency (u,p)} from \eqref{eq: discrete SDF} and we obtain the following error equations
\begin{equation}\label{eq: equazione errore}
    \begin{cases}
    \begin{aligned}
        \ah{\vh{u}{}}{\vh{v}{}}{} - \ah{\bbf{u}}{\vh{v}{}}{} + \bh{\vh{v}{}}{e^p}{} &= 0 \qquad &&\forall \vh{v}{} \in \bbf{X}_h,\\ 
        -\bh{\bbf{e}^u}{\ph{q}{}}{} + \sh{e^p}{\ph{q}{}} &= 0 \qquad &&\forall \ph{q}{} \in M_h. \\
    \end{aligned}
    \end{cases}
\end{equation}
Using the interpolation and discrete error definitions, system \eqref{eq: equazione errore} becomes:
\begin{equation}\label{eq: bh(vh, ehp)}
    \begin{cases}
    &\begin{split}
         \ah{\vh{u}{}}{\vh{v}{}}{} - \ah{\bbf{u}_I}{\vh{v}{}}{} &+ \ah{\bbf{u}_I}{\vh{v}{}}{} - \ah{\bbf{u}}{\vh{v}{}}{} \\ &+ b_h(\vh{v}{},\eph{e}{p}) + \bh{\vh{v}{}}{e^p_I}{}= 0 \qquad \forall \vh{v}{} \in \bbf{X}_h,\\
    \end{split}\\
    & - \bh{\evh{e}{u}}{\ph{q}{}}{} - \bh{\bbf{e}^u_I}{\ph{q}{}}{} + \sh{\eph{e}{p}}{\ph{q}{}} + \sh{e_I^p}{\ph{q}{}} = 0 \qquad \forall \ph{q}{} \in M_h.
\end{cases}
\end{equation}
Then, we test against $\vh{v}{} = \evh{e}{u}$ and $\ph{q}{} = \eph{e}{p}$, to obtain
\begin{equation}
    \begin{cases}
    &\ah{\vh{u}{}}{\evh{e}{u}}{} - \ah{\bbf{u}_I}{\evh{e}{u}}{} + \bh{\evh{e}{u}}{\eph{e}{p}}{} = \ah{\bbf{u}}{\evh{e}{u}}{} -\ah{\bbf{u}_I}{\evh{e}{u}}{} - \bh{\evh{e}{u}}{e^p_I}{}\\
    & -\bh{\evh{e}{u}}{\eph{e}{p}}{} + \sh{\eph{e}{p}}{\eph{e}{p}} = \bh{\bbf{e}^u_I}{\eph{e}{p}}{} - \sh{e^p_I}{\eph{e}{p}}
\end{cases}
\end{equation}
Summing up the two equations above, we obtain
\begin{equation}\label{eq: errore 2}
    \ah{\vh{u}{}}{\evh{e}{u}}{} - \ah{\bbf{u}_I}{\evh{e}{u}}{} + \sh{\eph{e}{p}}{\eph{e}{p}} = \underbrace{\ah{\bbf{u}}{\evh{e}{u}}{} - \ah{\bbf{u}_I}{\evh{e}{u}}{}}_{\bbf{R_1}} \underbrace{- \bh{e_I^p}{\evh{e}{u}}{}}_{\bbf{R_2}} + \underbrace{\bh{\bbf{e}^u_I}{\eph{e}{p}}{}}_{\bbf{R_3}} \underbrace{- \sh{e_I^p}{\eph{e}{p}}}_{\bbf{R_4}}.
\end{equation}
Now, for the monotonicity of $\ah{\cdot}{\cdot}{}$, see Lemma~\ref{lem: ah mon}, we have:
\begin{equation}
    \tnormxq{\evh{e}{u}}{} + |\eph{e}{p}|_J^2 \lesssim \bbf{R_1} + \bbf{R_2} + \bbf{R_3} + \bbf{R_4}.
\end{equation}
From the discrete inf-sup condition \eqref{eq:discrete infsup} and the first equation in \eqref{eq: bh(vh, ehp)} it is inferred that 
\begin{equation}\label{eq: errore p bound}
    \begin{split}
        \|\eph{e}{p}\|_{L^2(\Omega)} &\lesssim \sup_{\vh{v}{} \in \bbf{X}_h, \vh{v}{}\neq\bbf{0}} \frac{\bh{\vh{v}{}}{\eph{e}{p}}{}}{\tnormdiv{\vh{v}{}}{}} + |\eph{e}{p}|_J \\ & = \sup_{\vh{v}{} \in \bbf{X}_h, \vh{v}{} \neq \bbf{0}} \frac{\ah{\bbf{u}}{\vh{v}{}}{} - \ah{\bbf{u}_I}{\vh{v}{}}{} - \bh{\vh{v}{}}{e_I^p}{} - \ah{\vh{u}{}}{\vh{v}{}}{} + \ah{\bbf{u}_I}{\vh{v}{}}{} }{\tnormdiv{\vh{v}{}}{}} + |\eph{e}{p}|_J\\ & \lesssim \tnormx{\bbf{e}^u_I}{} + \tnormdiv{\evh{e}{u}}{} + \|e_I^p\|_{L^2(\Omega)}+ |\eph{e}{p}|_J.
    \end{split}
\end{equation}
Putting together \eqref{eq: errore 2}, \eqref{eq: errore p bound}, and proceeding as in the stability estimates, taking into account that we can bound the norm of the divergence of ${\evh{e}{u}}{}$ as in \eqref{eq: bound norm div}, we obtain that:
\begin{equation}\label{eq: convergence estimate 1}
\begin{split}
    \tnormdivq{\evh{e}{u}}{} + \tnormpq{\eph{e}{p}}{} \lesssim \bbf{R_1 + R_2 + R_3 + R_4} + \tnormxq{\bbf{e}^u_I}{} + \|e_I^p\|_{L^2(\Omega)}. 
\end{split}
\end{equation}
Now, we study the bound of $\tnormxq{\bbf{e}^u_I}{}$ and $\|e_I^p\|_{L^2(\Omega)}$.
Referring to Theorem \ref{thm: Approximation property}, we conclude that:
\begin{equation}
    \tnormxq{\bbf{e}^u_I}{} \lesssim h^{2t_\mcal{S}} |\bbf{u}^\mcal{S}|^2_{H^{t_\mcal{S}+1}(\Omega_\mcal{S})} + h^{2t_\mcal{D}} |\bbf{u}_\mcal{D}|^2_{H^{t_\mcal{D}+1}(\Omega_\mcal{D})} + h^{2t} |\bbf{u}|^2_{H^{t+1}(\Omega)},
\end{equation}
and
\begin{equation}
    \|e_I^p\|_{L^2(\Omega)} \lesssim h_K^{2(m_\mcal{S}+1)} |p^\mcal{S}|^2_{H^{m_\mcal{S}+1}(\Omega_\mcal{S})} + h_K^{2(m_\mcal{D}+1)} |p^\mcal{D}|^2_{H^{m_\mcal{D}+1}(\Omega_\mcal{D})}.
\end{equation}
Now we analyze the terms $\bbf{R_1, R_2, R_3}$ and $\bbf{R_4}$ separately. For $\bbf{R_1}$, we set
\begin{equation}
\begin{split}
    \bbf{R_1} &= \underbrace{\sum_{K \in \scr{K}_\mcal{S}} \int_K (g_S(|\bbf{D}(\bbf{u}^S)|)\bbf{D}(\bbf{u}^S) - g_S(|\bbf{D}(\bbf{u}^S_I)|)\bbf{D}(\bbf{u}_I^S)):\bbf{D}(\evh{e}{u^\mcal{S}}) dx}_{I_1} \\ & - \underbrace{\sum_{F \in \scr{F}_\mcal{S}} \int_F \averagel g_S(|\bbf{D}(\bbf{u}^S)|)\bbf{D}(\bbf{u}^S)\bbf{n} - g_S(|\bbf{D}(\bbf{u}^S_I)|)\bbf{D}(\bbf{u}^S_I)\bbf{n} \averager \cdot \jumpl \evh{e}{u^\mcal{S}} \jumpr ds}_{I_2} \\ & + \underbrace{\sum_{j=1}^{d-1} \sum_{F \in \Gamma} \int_F \rho^{-1} (\evh{e}{u^\mcal{S}} \cdot \bbf{t}_{\Gamma,j}) (\bbf{e}_I^{u^S} \cdot \bbf{t}_{\Gamma,j}) ds}_{I_3}  + \underbrace{\sum_{K \in \scr{K}_\mcal{D}} \int_K [(\bbf{K}^{-1}(g_D(|\bbf{u}^D|)\bbf{u}^D - g_D(|\bbf{u}^D_I|)\bbf{u}^D_I)]\cdot (\evh{e}{u^\mcal{D}}) dx}_{I_4} \\ & + \underbrace{\sum_{F \in \scr{F}_\mcal{S}} \int_F \sigma_S \jumpl \bbf{e}_I^{u^S} \jumpr \cdot \jumpl \evh{e}{u^\mcal{S}} \jumpr ds + \sum_{F \in \scr{F}_\mcal{D}} \int_F \sigma_D \jumpl \bbf{e}_I^{u^D} \jumpr_\bbf{n} \jumpl \evh{e}{u^\mcal{D}} \jumpr_\bbf{n} ds + \sum_{F \in \Gamma} \int_F \sigma_\Gamma \jumpl \bbf{e}_I^{u} \jumpr_\bbf{n} \jumpl \evh{e}{u} \jumpr_\bbf{n} ds, }_{I_5}.
\end{split}
\end{equation}
First, we consider $I_1$ and we apply the Lipschitz continuity of $g_S(|\cdot|)\cdot$, \eqref{eq: g_S bdd and Lip continuous}, Young's inequality and the approximation property of the $L^2$-interpolation, see \eqref{thm: Approximation property}, to obtain\\ 
\begin{equation} \label{eq: R1_1}
    I_1 \lesssim h^{2t_\mcal{S}} |\bbf{u}^S|_{H^{t_\mcal{S}+1}(\Omega_S)}^2 + \| \nabla_h \evh{e}{u^\mcal{S}} \|_{L^2(\Omega_S)}^2.
\end{equation}
For the bound of $I_2$, we proceed as before and we use also the trace inverse inequality to recover the norm of the interpolant on  $\Omega_S$, i.e.
\begin{equation}\label{eq: R1_2}
\begin{split}
    I_2 \lesssim h^{2t_\mcal{S}} |\bbf{u}^S|_{H^{t_\mcal{S}+1}(\Omega_S)}^2 + \sum_{F \in \scr{F}_\mcal{S}} h_K^{-1} \| \jumpl \evh{e}{u^\mcal{S}} \jumpr \|_{L^2(F)}.
\end{split}
\end{equation}
Now, we consider the interface terms and employ Young's inequality, the trace inequality, and the approximation properties of the $L^2$-interpolation to obtain
\begin{equation}\label{eq: R1_3}
    \begin{split}
        I_3 \lesssim \sum_{K: \partial K \cap \Gamma \neq \emptyset} h_K^{2t_\mcal{S}} |\bbf{u}^S|^2_{H^{t_\mcal{S}+1}(K)} + \sum_{F \in \Gamma} \rho^{-1}\|\evh{e}{u^\mcal{S}} \cdot \bbf{t}_{\Gamma,j}\|^2_{L^2(F)} \qquad \forall j=1,...,d-1.
    \end{split}
\end{equation}
For the terms relative to $\Omega_D$, as before, we apply the trace-inverse inequality, Young's inequality and the approximation properties of the interpolation. Additionally, we apply the Lipschitz continuity of $g_D(|\cdot|)\cdot$, see \eqref{eq: g_D bdd and Lip continuous}, to obtain 
\begin{equation}\label{eq: R1_4}
    \begin{split}
        I_4 \lesssim h^{2(t_\mcal{D} + 1)} |\bbf{u}^D|^2_{H^{t_\mcal{D}+1}(\Omega_D)} + \|\evh{e}{u^\mcal{D}}\|_{L^2(\Omega_D)}.
    \end{split}
\end{equation}
Finally, considering the last term contributing to $\bbf{R_1}$, we have
\begin{equation}
    \begin{split}
        I_5 &\lesssim h_K^{2t_\mcal{S}}|\bbf{u}^S|^2_{H^{t_\mcal{S}+1}} +  h_K^{2t_\mcal{D}} |\bbf{u}^D|^2_{H^{t_\mcal{D}+1}}(\Omega_D) + h_K^{2t} |\bbf{u}|^2_{H^{t+1}(\Omega)}
         \\ &
        + \sum_{F \in \scr{F}_\mcal{S}} \sigma_S \| \jumpl \evh{e}{u^\mcal{S}} \jumpr \|^2_{L^2(F)} + \sum_{F \in \scr{F}_\mcal{D}} \sigma_D \| \jumpl \evh{e}{u^\mcal{D}} \jumpr_\bbf{n} \|^2_{L^2(F)} + \sum_{F \in \Gamma} \sigma_\Gamma \| \jumpl \bbf{e}^{u}_h \jumpr_\bbf{n} \|^2_{L^2(F)}.
    \end{split}
\end{equation}
Now, we proceed to study the other terms of \eqref{eq: convergence estimate 1}. For $\bbf{R_2}, \bbf{R_3}$ and $\bbf{R_4}$ we apply the Cauchy-Schwarz and Young's inequalities for the volume integrals to obtain
\begin{align}
    \begin{split}
         \bbf{R_2} & \lesssim h^{2m_\mcal{S}+2}|p^S|_{H^{m_\mcal{S}+1}(\Omega_S)}^2 + h_K^{2m_\mcal{D}+2}|p^D|_{H^{m_\mcal{D}+1}(\Omega_D)}^2 + \| \nabla_h \cdot \evh{e}{u^\mcal{S}} \|^2_{L^2(\Omega_S)} + \| \nabla_h \cdot \evh{e}{u^\mcal{D}} \|_{L^2(\Omega_D)}^2 \\ & + \sum_{F \in \scr{F}_\mcal{S}} h_K^{-1} \| \jumpl \bbf{e}^u_h \jumpr \|^2_{L^2(F)} + \sum_{F \in \scr{F}_\mcal{D}} h_K^{-1} \| \jumpl \evh{e}{u^\mcal{D}} \jumpr_\bbf{n} \|^2_{L^2(F)} + \sum_{F \in \Gamma} h_K^{-1} \| \jumpl \bbf{e}^u_h \jumpr_\bbf{n} \|^2_{L^2(F)},
    \end{split} \\
    \bbf{R_3} & \lesssim h^{2t_\mcal{S}}| \bbf{u}^S |^2_{H^{t_\mcal{S}+1}(\Omega_S)} + h_K^{2t_\mcal{D}} | \bbf{u}^D|^2_{H^{t_\mcal{D}+1}(\Omega_D)} + 2(\| \eph{e}{p^\mcal{S}} \|^2_{L^2(\Omega_S)} + \| \eph{e}{p^\mcal{D}} \|^2_{L^2(\Omega_D)}), \\
    \begin{split}
        \bbf{R_4} & \lesssim h^{2(m_\mcal{S}+1)} |p^S|^2_{H^{m_\mcal{S}+1}(\Omega_S)} + h^{2(m_\mcal{D}+1)} |p^D|^2_{H^{m_\mcal{D}+1}(\Omega_D)} + \sum_{F \in \scr{F}_\mcal{S}^i} \xi_S \| \jumpl \eph{e}{p^\mcal{S}} \jumpr \|_{L^2(F)}^2 \\ &+ \sum_{F \in \scr{F}_\mcal{D}^i} \xi_D \| \jumpl \eph{e}{p^\mcal{D}} \jumpr \|_{L^2(F)}^2 + \sum_{F \in \Gamma} \xi_\Gamma \| \jumpl \eph{e}{p} \jumpr \|_{L^2(F)}^2.
    \end{split}
\end{align}
Finally, putting together all the previous bounds we obtain:
\begin{equation}
    \begin{split}
        \tnormdivq{\evh{e}{u}}{} + \tnormpq{\eph{e}{p}}{} \lesssim& h^{2t_\mcal{S}}|\bbf{u}^S|_{H^{t_\mcal{S}+1}(\Omega_S)}^2 + h^{2t_\mcal{D}}|\bbf{u}^D|_{H^{t_\mcal{D}+1}(\Omega_D)}^2 + h^{2t}|\bbf{u}|_{H^{t+1}(\Omega)}^2 \\ &+ h^{2m_\mcal{S}+2} |p^\mcal{S}|_{H^{m_\mcal{S}+1}(\Omega_\mcal{S})} + h^{2m_\mcal{D}+2} |p^\mcal{D}|_{H^{m_\mcal{D}+1}(\Omega_\mcal{D})},
    \end{split}
\end{equation}
which implies
\begin{equation}
    \begin{split}
        \tnormeq{(\evh{e}{u},\eph{e}{p})}{} \lesssim& h^{2t_\mcal{S}}|\bbf{u}^S|_{H^{t_\mcal{S}+1}(\Omega_S)}^2 + h^{2t_\mcal{D}}|\bbf{u}^D|_{H^{t_\mcal{D}+1}(\Omega_D)}^2 + h^{2t}|\bbf{u}|_{H^{t+1}(\Omega)}^2 \\ &+ h^{2m_\mcal{S}+2} |p^\mcal{S}|_{H^{m_\mcal{S}+1}(\Omega_\mcal{S})} + h^{2m_\mcal{D}+2} |p^\mcal{D}|_{H^{m_\mcal{D}+1}(\Omega_\mcal{D})},
    \end{split}
\end{equation}
with $t = \min\{t_\mcal{S}, t_\mcal{D}\}$, and the proof is completed.
\end{proof}

\section{Numerical results}\label{sec: numerical results}
In this section, we present numerical experiments to verify the error bounds established in the previous section and to assess the practical performance of the proposed discretization. The numerical implementations are carried out in the open-source Lymph MATLAB library \cite{lymph}, implementing the PolyDG method for multi-physics in $2D$ domains.
\subsection{Test case 1: linear constitutive law (Newtonian fluids)}\label{sec: testcase1}
We consider a numerical convergence test for a Newtonian fluid, on a square domain $\Omega = (-1, 1) \times (-1, 1)$ with an interface $\Gamma = \{ x=0 \}$. On the left part of the domain, $\Omega_S = (-1, 0) \times (-1, 1)$, we solve the Stokes equation, while on the right part of the domain, $\Omega_D = (0, 1) \times (-1, 1)$, we solve the Darcy equation. We consider the manufactured solution:
\begin{equation}
    \bbf{u}_\mcal{S}^{ex}(\bbf{x}) = \left[
    \begin{array}{c}
         x^3 \pi \cos(\pi y) \sin(\pi y) \\
         -\frac{3}{2} x^2 \sin(\pi y)^2
    \end{array}
    \right], \ \ \ p_\mcal{S}^{ex}(\bbf{x}) = e^y \cos(\pi x), 
\end{equation}
\begin{equation}
    \bbf{u}_\mcal{D}^{ex}(\bbf{x}) = \left[
    \begin{array}{c}
         x^2 \pi \cos(\pi y) \sin(\pi y) \\
         -x \sin(\pi y)^2
    \end{array}
    \right], \ \ \ p_\mcal{D}^{ex}(\bbf{x}) = e^y \cos(\pi x) - 6 \pi x^2 \cos(\pi y) \sin(\pi y). 
\end{equation}
The problem is solved with polynomial degree $l = 2,3,4$ and on a sequence of successively finer polygonal meshes with diameter $h = 0.2485, 0.1811, 0.1301, 0.0902, 0.0640, 0.0457, 0.0325, 0.0228$.
In Figure~\ref{fig:test1} we report the computed error estimates in the energy norm $\tnorme{(\cdot,\cdot)}$ and visually highlight the square root of the numerical order of convergence. The results are in agreement with Theorem \ref{thm: error estimate} as the error goes to zero with the predicted algebraic rate $h^l$, as $h$ goes to zero.
\subsection{Test case 2: Carreau constitutive law}
We next present a numerical convergence test for a non-Newtonian fluid, with viscosity modeled following the Carreau constitutive law. We consider a domain defined as in Section~\ref{sec: testcase1}. We took a similar manufactured solution:
\begin{align}
    &\bbf{u}_\mcal{S}^{ex}(\bbf{x}) = \left[
    \begin{array}{c}
         x^3 \pi \cos(\pi y) \sin(\pi y) \\
         -\frac{3}{2} x^2 \sin(\pi y)^2
    \end{array}
    \right], \ \ \ p_\mcal{S}^{ex}(\bbf{x}) = e^y \cos(\pi x), \\ 
    &\bbf{u}_\mcal{D}^{ex}(\bbf{x}) = \left[
    \begin{array}{c}
         x^2 \pi \cos(\pi y) \sin(\pi y) \\
         -x \sin(\pi y)^2
    \end{array}
    \right], \ \ \ p_\mcal{D}^{ex}(\bbf{x}) = e^y \cos(\pi x) - (g_\mcal{S}(|\bbf{D}(\vh{u}{S})|)\bbf{D}(\vh{u}{S})\bbf{n})\bbf{n},
\end{align}
assuming the following constitutive equations 
\begin{align}
    &g_\mcal{S}(|\bbf{D}(\vh{u}{S})|) = \underline{\nu_\mcal{S}}+(\overline{\nu_\mcal{S}}-\underline{\nu_\mcal{S}})\left(1 + \frac{1}{2}|\bbf{D}(\vh{u}{S})|^2)\right)^{-\frac{1}{4}},\\
    &g_\mcal{D}(|\vh{u}{S}|) = \underline{\nu_\mcal{D}}+(\overline{\nu_\mcal{D}}-\underline{\nu_\mcal{D}})\left(1 + \frac{1}{2}|\vh{u}{D}|^2\right)^{-\frac{1}{4}},
\end{align}
with $\underline{\nu_\mcal{S}} = \underline{\nu_\mcal{D}} = 0.001$ and $\overline{\nu_\mcal{S}} = \overline{\nu_\mcal{D}} = 0.5$ \cite{SDFDG}.
The forcing terms and the boundary conditions are then set accordingly. The problem is solved with $l = 2, 3, 4$ on the same sequence of meshes of Section~\ref{sec: testcase1}. For the solution of the non-Newtonian system, we adopt a fixed-point iteration scheme with tolerance $10^{-10}$. In Figure~\ref{fig:test2} we report the computed error estimates of the energy norm $\tnorme{(\cdot,\cdot)}$ and visually highlight the root mean square of the numerical order of convergence. Again, the results are in agreement with the theory defined in Section~\ref{sec: error estimates} as the error goes to zero with the predicted algebraic rate $h^l$, as $h$ goes to zero.

\begin{figure}[H]
   \begin{subfigure}[b]{0.5\textwidth}
    \centering
    \resizebox{\textwidth}{!}{\definecolor{mycyan}{rgb}{0.00000,1.00000,1.00000}%
\definecolor{myblue}{HTML}{1E2A78}%
\definecolor{myred}{HTML}{D7265B}%
\definecolor{myorange}{HTML}{F49D37}%
\pgfplotsset{
  log x ticks with fixed point/.style={
      xticklabel={
        \pgfkeys{/pgf/fpu=true}
        \pgfmathparse{exp(\tick)}%
        \pgfmathprintnumber[fixed  zerofill, precision=2]{\pgfmathresult}
        \pgfkeys{/pgf/fpu=false}
      }
  }
}
\begin{tikzpicture}

\begin{axis}[%
xticklabel style={rotate=45, anchor=east},
width=3.875in,
height=2.36in,
at={(2.6in,1.099in)},
scale only axis,
xmode=log,
xmin=0.0228,
xmax=0.2485,
xminorticks=true,
xlabel = {$h$ [-]},
ylabel = {},
log x ticks with fixed point,
ymode=log,
ymin=1e-7,
ymax=1e1,
yminorticks=true,
axis background/.style={fill=white},
title style={font=\bfseries},
title={},
xmajorgrids,
xminorgrids,
ymajorgrids,
yminorgrids,
legend pos = south east,
legend style={legend cell align=left, align=left, draw=white!15!black}
]
              
\addplot [color=myred, line width=1.0pt]
  table[row sep=crcr]{%
0.2485  1.0468\\
0.1811  0.6111\\
0.1301  0.2866\\
0.0902  0.1071\\
0.0640  0.0446\\
0.0457  0.0231\\
0.0325  0.0113\\
0.0228  0.0054\\
};
\addlegendentry{$l = 2$}

\fill[myred] (0.2485,1.0468) circle (2pt);
\fill[myred] (0.1811,0.6111) circle (2pt);
\fill[myred] (0.1301,0.2866) circle (2pt);
\fill[myred] (0.0902,0.1071) circle (2pt);
\fill[myred] (0.0640,0.0446) circle (2pt);
\fill[myred] (0.0457,0.0231) circle (2pt);
\fill[myred] (0.0325,0.0113) circle (2pt);
\fill[myred] (0.0228,0.0054) circle (2pt);

\addplot [color=myblue, line width=1.0pt]
  table[row sep=crcr]{%
0.2485  0.1047\\
0.1811  0.0382\\
0.1301  0.0134\\
0.0902  0.0043\\
0.0640  0.0015\\
0.0457  5.2882e-04\\
0.0325  1.8984e-04\\
0.0228  6.8765e-05\\
};
\addlegendentry{$l = 3$}

\fill[myblue] (0.2485,0.1047) circle (2pt);
\fill[myblue] (0.1811,0.0382) circle (2pt);
\fill[myblue] (0.1301,0.0134) circle (2pt);
\fill[myblue] (0.0902,0.0043) circle (2pt);
\fill[myblue] (0.0640,0.0015) circle (2pt);
\fill[myblue] (0.0457,5.2882e-04) circle (2pt);
\fill[myblue] (0.0325,1.8984e-04) circle (2pt);
\fill[myblue] (0.0228,6.8765e-05) circle (2pt);

\addplot [color=myorange, line width=1.0pt]
  table[row sep=crcr]{%
0.2485  0.0089\\
0.1811  0.0024\\
0.1301  6.2836e-04\\
0.0902  1.5271e-04\\
0.0640  3.7791e-05\\
0.0457  9.6400e-06\\
0.0325  2.5079e-06\\
0.0228  6.4392e-07\\
};
\addlegendentry{$l = 4$}

\fill[myorange] (0.2485,0.0089) circle (2pt);
\fill[myorange] (0.1811,0.0024) circle (2pt);
\fill[myorange] (0.1301,6.2836e-04) circle (2pt);
\fill[myorange] (0.0902,1.5271e-04) circle (2pt);
\fill[myorange] (0.0640,3.7791e-05) circle (2pt);
\fill[myorange] (0.0457,9.6400e-06) circle (2pt);
\fill[myorange] (0.0325,2.5079e-06) circle (2pt);
\fill[myorange] (0.0228,6.4392e-07) circle (2pt);

\node[right, align=left, text=black, font=\footnotesize]
at (axis cs:0.0328,0.0035) {$2.1222$};

\addplot [color=black, line width=1.5pt]
  table[row sep=crcr]{%
2.51e-02   0.0025\\
0.0325   0.0025\\
0.0325   0.0043\\
2.51e-02   0.0025\\
};

\node[right, align=left, text=black, font=\footnotesize]
at (axis cs:0.0328, 7e-05) {$3.0721$};

\addplot [color=black, line width=1.5pt]
  table[row sep=crcr]{%
2.51e-02   4e-05\\
0.0325   4e-05\\
0.0325   8.8467e-05\\
2.51e-02   4e-05\\
};

\node[right, align=left, text=black, font=\footnotesize]
at (axis cs:0.0328, 5e-07) {$3.9981$};

\addplot [color=black, line width=1.5pt]
  table[row sep=crcr]{%
2.51e-02   3e-07\\
0.0325   3e-07\\
0.0325   8.4284e-07\\
2.51e-02   3e-07\\
};

\end{axis}
\end{tikzpicture}%
}
        \caption[]%
         {Test case 1: Computed errors $\tnorme{(\bbf{e}^u,e^p)}$ and the root mean square of their numerical order of convergence.}    
        \label{fig:test1}
    \end{subfigure}\hspace{3mm}
    \begin{subfigure}[b]{0.5\textwidth}
      \centering
    \resizebox{\textwidth}{!}{\definecolor{mycyan}{rgb}{0.00000,1.00000,1.00000}%
\definecolor{myblue}{HTML}{1E2A78}%
\definecolor{myred}{HTML}{D7265B}%
\definecolor{myorange}{HTML}{F49D37}%
\pgfplotsset{
  log x ticks with fixed point/.style={
      xticklabel={
        \pgfkeys{/pgf/fpu=true}
        \pgfmathparse{exp(\tick)}%
        \pgfmathprintnumber[fixed  zerofill, precision=2]{\pgfmathresult}
        \pgfkeys{/pgf/fpu=false}
      }
  }
}
\begin{tikzpicture}

\begin{axis}[%
xticklabel style={rotate=45, anchor=east},
width=3.875in,
height=2.36in,
at={(2.6in,1.099in)},
scale only axis,
xmode=log,
xmin=0.0228,
xmax=0.2485,
xminorticks=true,
xlabel = {$h$ [-]},
ylabel = {},
log x ticks with fixed point,
ymode=log,
ymin=1e-7,
ymax=1e1,
yminorticks=true,
axis background/.style={fill=white},
title style={font=\bfseries},
title={},
xmajorgrids,
xminorgrids,
ymajorgrids,
yminorgrids,
legend pos = south east,
legend style={legend cell align=left, align=left, draw=white!15!black}
]
              
\addplot [color=myred, line width=1.0pt]
  table[row sep=crcr]{%
0.2485  0.6275\\
0.1811  0.3216\\
0.1301  0.1567\\
0.0902  0.0717\\
0.0640  0.0323\\
0.0457  0.0159\\
0.0325  0.0079\\
0.0228  0.0040\\
};
\addlegendentry{$l = 2$}

\fill[myred] (0.2485,0.6275) circle (2pt);
\fill[myred] (0.1811,0.3216) circle (2pt);
\fill[myred] (0.1301,0.1567) circle (2pt);
\fill[myred] (0.0902,0.0717) circle (2pt);
\fill[myred] (0.0640,0.0323) circle (2pt);
\fill[myred] (0.0457,0.0159) circle (2pt);
\fill[myred] (0.0325,0.0079) circle (2pt);
\fill[myred] (0.0228,0.0040) circle (2pt);

\addplot [color=myblue, line width=1.0pt]
  table[row sep=crcr]{%
0.2485  0.1031\\
0.1811  0.0372\\
0.1301  0.0133\\
0.0902  0.0044\\
0.0640  0.0016\\
0.0457  5.5282e-04\\
0.0325  1.9915e-04\\
0.0228  7.6223e-05\\
};
\addlegendentry{$l = 3$}

\fill[myblue] (0.2485,0.1031) circle (2pt);
\fill[myblue] (0.1811,0.0372) circle (2pt);
\fill[myblue] (0.1301,0.0133) circle (2pt);
\fill[myblue] (0.0902,0.0044) circle (2pt);
\fill[myblue] (0.0640,0.0016) circle (2pt);
\fill[myblue] (0.0457,5.5282e-04) circle (2pt);
\fill[myblue] (0.0325,1.9915e-04) circle (2pt);
\fill[myblue] (0.0228,7.6223e-05) circle (2pt);

\addplot [color=myorange, line width=1.0pt]
  table[row sep=crcr]{%
0.2485  0.0090\\
0.1811  0.0025\\
0.1301  6.1823e-04\\
0.0902  1.4640e-04\\
0.0640  3.6236e-05\\
0.0457  9.2331e-06\\
0.0325  2.4268e-06\\
0.0228  6.2568e-07\\
};
\addlegendentry{$l = 4$}

\fill[myorange] (0.2485,0.0090) circle (2pt);
\fill[myorange] (0.1811,0.0025) circle (2pt);
\fill[myorange] (0.1301,6.1823e-04) circle (2pt);
\fill[myorange] (0.0902,1.4640e-04) circle (2pt);
\fill[myorange] (0.0640,3.6236e-05) circle (2pt);
\fill[myorange] (0.0457,9.2331e-06) circle (2pt);
\fill[myorange] (0.0325,2.4268e-06) circle (2pt);
\fill[myorange] (0.0228,6.2568e-07) circle (2pt);

\node[right, align=left, text=black, font=\footnotesize]
at (axis cs:0.0328,0.0035) {$2.1222$};

\addplot [color=black, line width=1.5pt]
  table[row sep=crcr]{%
2.51e-02   0.0025\\
0.0325   0.0025\\
0.0325   0.0043\\
2.51e-02   0.0025\\
};

\node[right, align=left, text=black, font=\footnotesize]
at (axis cs:0.0328, 7e-05) {$3.0261$};

\addplot [color=black, line width=1.5pt]
  table[row sep=crcr]{%
2.51e-02   4e-05\\
0.0325   4e-05\\
0.0325   8.7421e-05\\
2.51e-02   4e-05\\
};

\node[right, align=left, text=black, font=\footnotesize]
at (axis cs:0.0328, 5e-07) {$3.9481$};

\addplot [color=black, line width=1.5pt]
  table[row sep=crcr]{%
2.51e-02   3e-07\\
0.0325   3e-07\\
0.0325   8.3202e-07\\
2.51e-02   3e-07\\
};

\end{axis}
\end{tikzpicture}%
}
        \caption[]%
        {Test case 2: Computed errors $\tnorme{(\bbf{e}^u,e^p)}$ and the root mean square of their numerical order of convergence.}    
        \label{fig:test2}
    \end{subfigure}\hfill
\end{figure}

\section{Conclusions}\label{sec: Conclusion}
We have developed and analyzed a polytopal discontinuous Galerkin method for the numerical approximation of a coupled non-Newtonian Stokes-Darcy system describing the interaction between a non-Newtonian free-flow fluid and a non-Newtonian porous-medium flow. A complete well-posedness analysis has been established for the continuous problem. Within the framework of generalized inf-sup theory, we also prove the well-posedness, stability, and convergence of the proposed numerical method, and derived corresponding error estimates.  The theoretical estimates were validated through two numerical convergence tests using manufactured solutions: the first assuming a linear (Newtonian) viscosity, and the second assuming a nonlinear (non-Newtonian) viscosity following the Carreau model, which is commonly used for shear-thinning non-Newtonian fluids. A natural extension of this work would be to incorporate more general nonlinear viscosity laws, such as a Forchheimer term in the porous medium model, as considered in \cite{SDFDG}, used to predict high velocity flow in porous media \cite{GiraultWheeler}, or a Carreau--Yasuda model \cite{CarreauYasuda_Santesarti}, employed in hemodynamic applications \cite{CarreauYasuda}. Further directions include the analysis of the non-Newtonian coupled problem under more general mesh assumptions, allowing an $hp$ analysis that also accounts for the dependence on the polynomial degree.

\bibliographystyle{abbrv}
\bibliography{bibliography}

@incollection{AntoniettiCangianiCollisDongGeorgoulisGianiHouston2016,
  author    = {Antonietti, Paola F. and Cangiani, Andrea and Collis, James and Dong, Zhaonan and Georgoulis, Emmanuil H. and Giani, Stefano and Houston, Paul},
  title     = {Review of Discontinuous {G}alerkin Finite Element Methods for Partial Differential Equations on Complicated Domains},
  booktitle = {Building Bridges: Connections and Challenges in Modern Approaches to Numerical Partial Differential Equations},
  series     = {Lecture Notes in Computational Science and Engineering},
  volume     = {114},
  pages      = {279--308},
  publisher  = {Springer},
  year       = {2016},
  doi        = {10.1007/978-3-319-41640-3_8}
}

@book{CangianiDongGeorgoulisHouston2017,
  author    = {Cangiani, Andrea and Dong, Zhaonan and Georgoulis, Emmanuil H. and Houston, Paul},
  title     = {hp-Version Discontinuous {G}alerkin Methods on Polygonal and Polyhedral Meshes},
  series     = {SpringerBriefs in Mathematics},
  publisher  = {Springer},
  address    = {Cham},
  year       = {2017},
  isbn       = {978-3-319-67672-2},
  doi        = {10.1007/978-3-319-67673-9}
}

@article{AntoniettiMascottoVeraniZonca2022,
  author  = {Antonietti, Paola F. and Mascotto, Lorenzo and Verani, Marco and Zonca, Stefano},
  title   = {Stability Analysis of Polytopic Discontinuous {G}alerkin Approximations of the {S}tokes Problem with Applications to Fluid--Structure Interaction Problems},
  journal = {Journal of Scientific Computing},
  volume   = {90},
  number   = {23},
  year     = {2022},
  doi      = {10.1007/s10915-021-01695-6}
}

@article{AntoniettiFacciolaRussoVerani2019,
  author  = {Antonietti, Paola F. and Facciol{\`a}, Chiara and Russo, Alessandro and Verani, Marco},
  title   = {Discontinuous {G}alerkin Approximation of Flows in Fractured Porous Media on Polytopic Grids},
  journal = {SIAM Journal on Scientific Computing},
  volume   = {41},
  number   = {1},
  pages    = {A109--A138},
  year     = {2019},
  doi      = {10.1137/17M1138194}
}

@article{Antonietti2020,
  author  = {Antonietti, Paola F.  and Facciol\`a, C. and Verani, M.},
  title   = {Polytopic Discontinuous {G}alerkin Methods for the Numerical Modelling of Flow in Porous Media with Networks of Intersecting Fractures},
  journal = {Journal of Computational and Applied Mathematics},
  year    = {2021},
  volume   = {393},
  pages    = {113506}
}

@article{YeZhang2020,
  author  = {Ye, X. and Zhang, S.},
  title   = {A Conforming Discontinuous {G}alerkin Finite Element Method for the {S}tokes Problem on Polytopal Meshes},
  journal = {Journal of Scientific Computing},
  year    = {2021},
  volume   = {87},
  pages    = {1--26}
}

@article{LipnikovVassilevYotov,
  author  = {Lipnikov, K. and Vassilev, D. and Yotov, I.},
  title   = {Discontinuous {G}alerkin and Mimetic Finite Difference Methods for Coupled {S}tokes--Darcy Flows on Polygonal and Polyhedral Grids},
  journal = {Numerical Methods for Partial Differential Equations},
  year    = {2014},
  volume   = {30},
  number   = {2},
  pages    = {515--543}
}

@article{SDFDG,
  title={{The Well-Posedness of Discontinuous {G}alerkin Approximation for the Non-Newtonian {S}tokes-Darcy-Forchheimer Coupling System}},
  author={Jingyan, Hu and Guanyu, Zhou},
  journal={Journal of Scientific Computing},
  volume={97},
  year={2023},
  publisher={Springer}
}

@article{MappedCoercivity,
  title = {{The concept of mapped coercivity for nonlinear operators in Banach spaces}},
    journal = {Journal of Functional Analysis},
    volume = {289},
    number = {3},
    pages = {110893},
    year = {2025},
    issn = {0022-1236},
    doi = {https://doi.org/10.1016/j.jfa.2025.110893},
    url = {https://www.sciencedirect.com/science/article/pii/S0022123625000758},
    author = {Becker, R. and Braack, M.}
    }

@article{Stability,
  title={{Stability Analysis of Polytopic Discontinuous {G}alerkin Approximations of the {S}tokes Problem with Applications to Fluid–Structure Interaction Problems}},
  author={Antonietti, Paola F.  and Mascotto, L. and Verani, M. and Zonca, S.},
  journal={Journal of Scientific Computing},
  volume={90},
  year={2021},
  publisher={Springer}
}

@book{SobolevEmbedding,
  author       = {Di Pietro, D.A. and Ern, A.} ,
  title        = {{Mathematical Aspects of Discontinuous {G}alerkin Methods}},
  publisher = {Springer},
  year         = 2012,
  isbn      = {978-3-642-22979-4}
}

@book{HighOrderpolynomialapprox,
  author       = {Di Pietro, D.A. and Droniou, J.} ,
  title        = {{The Hybrid High-Order Method for Polytopal Meshes}},
  publisher = {Springer},
  year         = 2020,
  isbn      = {978-3-030-37202-6}
}

@article{Leray-Lions,
  title={{A hybrid high-order method for Leray-Lions elliptic equations on general meshes}},
  author={Di Pietro, D.A. and Droniou, J.},
  journal={Mathematics of Computations},
  volume={86},
  number={307},
  pages={2159-2191},
  year={2017},
  publisher={Springer}
}

@article{AntoniettiGianiHouston_hpCompositeDG,
author = {Antonietti, Paola F.  and Giani, S. and Houston, P.},
title = {{$hp$-Version Composite Discontinuous {G}alerkin Methods for Elliptic Problems on Complicated Domains}},
journal = {SIAM Journal on Scientific Computing},
volume = {35},
number = {3},
pages = {A1417-A1439},
year = {2013},
doi = {https://doi.org/10.1137/120877246},
}

@article{BassiBottiColomboDiPietroTesini_flexibilityagglomeration,
title = {{On the flexibility of agglomeration based physical space discontinuous {G}alerkin discretizations}},
journal = {Journal of Computational Physics},
volume = {231},
number = {1},
pages = {45-65},
year = {2012},
issn = {0021-9991},
doi = {https://doi.org/10.1016/j.jcp.2011.08.018},
author = {Bassi, F. and Botti, L. and Colombo, A. and {Di Pietro}, D.A. and Tesini, P.}
}

@Article{Botti.Mascotto:26,
  author        = {Botti, M. and Mascotto, L.},
  title         = {{S}obolev--{P}oincar\'e inequalities for piecewise {$W^{1,p}$} functions over general polytopic meshes},
  year          = {2026},
  journal  = {SIAM Journal on Numerical Analysis},
  note = {Accepted for publication},
  arxiv         = {2504.03449},
  eprint        = {2504.03449},
  archiveprefix = {arXiv},
  primaryclass  = {math.NA}
}

@article{CangianiDongHoustonGeorgoulisHouston,
  author       = {Cangiani, A. and Dong, Z. and Georgoulis, E.H. and Houston, P.} ,
  title        = {{$hp$-version discontinuous {G}alerkin methods on polygonal and polyhedral meshes}},
  journal = {Mathematical Models and Methods in Applied Sciences},
  volume = {24},
    number = {10},
    pages = {2009-2041},
  year         = 2014
}

@article{NumericalAnalysisCoupStokesDarcyFlowsIndustrialFitrations,
author = {Hanspal, N. and Waghode, A. and Nassehi, V. and Wakeman, R.},
title = {{Numerical Analysis of Coupled {S}tokes/Darcy Flows in Industrial Filtrations}},
journal = {Transport in Porous Media},
volume = {64},
pages = {73-101},
doi = {10.1007/s11242-005-1457-3},
year = {2006},
}

@article{ErvinJenkinsSunFEM,
author = {Ervin, V. J. and Jenkins, E. W. and Sun, S.},
title = {{Coupled Generalized Nonlinear {S}tokes Flow with Flow through a Porous Medium}},
journal = {SIAM Journal on Numerical Analysis},
volume = {47},
number = {2},
pages = {929-952},
year = {2009},
doi = {https://doi.org/10.1137/070708354},
}

@article{ErvinJenkinsSunMORTARFEM,
title = {{Coupling nonlinear {S}tokes and Darcy flow using mortar finite elements}},
journal = {Applied Numerical Mathematics},
volume = {61},
number = {11},
pages = {1198-1222},
year = {2011},
issn = {0168-9274},
doi = {https://doi.org/10.1016/j.apnum.2011.08.002},
author = {Ervin, V.J. and Jenkins, E.W. and Sun,S.},
}

@article{GiraultRiviere,
author = {Girault, V. and Rivi\`{e}re, B.},
title = {{DG Approximation of Coupled Navier–{S}tokes and Darcy Equations by Beaver–Joseph–Saffman Interface Condition}},
journal = {SIAM Journal on Numerical Analysis},
volume = {47},
number = {3},
pages = {2052-2089},
year = {2009},
doi = {https://doi.org/10.1137/070686081},

}

@article{RiviereSD,
author = {Rivi\`{e}re, B.},
title = {{Analysis of a Discontinuous Finite Element Method for the Coupled {S}tokes and Darcy Problems}},
journal = {J Sci Comput},
volume = {22},
pages = {479-500},
year = {2005},
doi = {https://doi.org/10.1007/s10915-004-4147-3},
}

@article{GiraultWheeler,
  author    = {Girault, V. and Wheeler, M. F.},
  title     = {{Numerical discretization of a Darcy--Forchheimer model}},
  journal   = {Numerische Mathematik},
  year      = {2008},
  volume    = {110},
  number    = {2},
  pages     = {161-198},
  doi       = {https://doi.org/10.1007/s00211-008-0157-7},
}

@article{ChowCarey,
author = {Chow, S.-S. and Carey, G. F.},
title = {{Numerical approximation of generalized Newtonian fluids using Powell–Sabin–Heindl elements: I. theoretical estimates}},
journal = {International Journal for Numerical Methods in Fluids},
volume = {41},
number = {10},
pages = {1085-1118},
doi = {https://doi.org/10.1002/fld.480},
year = {2003}
}

@article{Discacciati_ApplicationThesis,
  title={{Domain decomposition methods for the coupling of surface and groundwater flows}},
  author={Distaccati, M.},
  journal={Ph.D. thesis, Ecole Polytechnique Federale de Sausanne, Sausanne, Switzerland},
  year={2024},
}

@article{Cesmelioglu_Applications,
title = {{An embedded–hybridized discontinuous {G}alerkin method for the coupled {S}tokes–Darcy system}},
author = {Cesmelioglu, A. and Rhebergen, S. and Wells, G.N.},
journal = {{Journal of Computational and Applied Mathematics}},
volume = {367},
pages = {112476},
year = {2020},
issn = {0377-0427},
doi = {https://doi.org/10.1016/j.cam.2019.112476},
}

@article{BrezziMarini,
author = {Arnold, D.N. and Brezzi, F. and Cockburn, B. and Marini, L.D.},
title = {{Unified Analysis of Discontinuous {G}alerkin Methods for Elliptic Problems}},
journal = {SIAM Journal on Numerical Analysis},
volume = {39},
number = {5},
pages = {1749-1779},
year = {2002},
doi = {https://doi.org/10.1137/S0036142901384162}
}

@article{DistaccatiMiglioQuarteroni,
author = {Discacciati, M. and Miglio, E. and Quarteroni, A.},
title = {Mathematical and numerical models for coupling surface and groundwater flows},
journal = {Applied Numerical Mathematics},
volume = {43},
number = {1},
pages = {57-74},
year = {2002},
doi = {https://doi.org/10.1016/S0168-9274(02)00125-3},
}

@article{DistaccatiQuarteroni,
  title={{Navier-{S}tokes/Darcy coupling: modeling, analysis, and numerical approximation}},
  author={Discacciati, M. and Quarteroni, A. and others},
  journal={Revista Matemática Complutense},
  volume={22},
  number={2},
  pages={315--426},
  year={2009}
}

@article{LiGaoZhangChen_SDD,
author = {Rui Li and Yali Gao and Chen-Song Zhang and Zhangxin Chen},
title = {{A {S}tokes–Darcy–Darcy model and its discontinuous {G}alerkin method on polytopic grids}},
journal = {Journal of Computational Physics},
volume = {501},
pages = {112780},
year = {2024},
issn = {0021-9991},
doi = {https://doi.org/10.1016/j.jcp.2024.112780},
}

@book{RiviereDG,
author = {Rivi\`{e}re, B.},
title = {{Discontinuous {G}alerkin Methods for Solving Elliptic and Parabolic Equations}},
publisher = {Society for Industrial and Applied Mathematics},
year = {2008},
doi = {10.1137/1.9780898717440},
}

@article{JagerMikelic_BJS,
 author = {Jäger, W. and Mikelić, A.},
 title = {{On the Interface Boundary Condition of Beavers, Joseph, and Saffman}},
 journal = {SIAM Journal on Applied Mathematics},
 number = {4},
 pages = {1111--1127},
 publisher = {Society for Industrial and Applied Mathematics},
 volume = {60},
 year = {2000}
}

@article{BurmanHansbo_stabilizedSD,
author = {Burman, E. and Hansbo, P.},
title = {{A unified stabilized method for {S}tokes’ and Darcy's equations}},
journal = {Journal of Computational and Applied Mathematics},
volume = {198},
number = {1},
pages = {35-51},
year = {2007},
issn = {0377-0427},
doi = {https://doi.org/10.1016/j.cam.2005.11.022},
}

@article{RiviereYotov,
 author = {Rivi\`{e}re, B. and Yotov, I.},
 journal = {SIAM Journal on Numerical Analysis},
 number = {5},
 pages = {1959--1977},
 publisher = {Society for Industrial and Applied Mathematics},
 title = {{Locally Conservative Coupling of {S}tokes and Darcy Flows}},
 volume = {42},
 year = {2005}
}

@article{badia_quaini_quarteroni_2008, title={{Coupling Biot and Navier-{S}tokes problems for fluid-poroelastic structure interaction}},
author={Badia, S. and Quaini, A. and Quarteroni, A.}, 
journal = {Technical report, Universitat Politcnica de Catalunya},
year={2008}}

@phdthesis{Quaini_thesis,
  author = {Quaini, A.},
  title = {Algorithms for Fluid-Structure Interaction Problems Arising in Hemodynamics},
  school = {Ecole Polytechnique Fédérale de Lausanne, Switzerland},
  year = {2008}
}

@BOOK{book:nonnewtonian,
	author = {Chhabra, R.P. and Patel, Swati A.},
	title = {Non-Newtonian Flow and Applied Rheology, Third Edition},
	year = {2025},
	journal = {Non-Newtonian Flow and Applied Rheology: Engineering Applications, Third Edition},
	pages = {1 – 562},
	doi = {10.1016/C2022-0-01444-6},
	publisher={Elsevier}
}

@article{DUNN1995689,
title = {Fluids of differential type: Critical review and thermodynamic analysis},
journal = {International Journal of Engineering Science},
volume = {33},
number = {5},
pages = {689-729},
year = {1995},
issn = {0020-7225},
doi = {https://doi.org/10.1016/0020-7225(94)00078-X},
author = {J.E. Dunn and K.R. Rajagopal}
}

@article{CarreauYasuda,
    author = {Boyd, J. and Buick, J.M. and Green, S.},
    title = {{Analysis of the Casson and Carreau-Yasuda non-Newtonian blood models in steady and oscillatory flows using the lattice Boltzmann method}},
    journal = {Physics of Fluids},
    volume = {19},
    number = {9},
    pages = {093103},
    year = {2007},
    doi = {https://doi.org/10.1063/1.2772250},
}

@article{CarreauYasuda_Santesarti,
title = {{A quasi-analytical solution for Carreau–Yasuda-type shear-thinning flows in slightly tapered pipes using a truncated power-law model}},
journal = {Journal of Non-Newtonian Fluid Mechanics},
volume = {349},
pages = {105569},
year = {2026},
issn = {0377-0257},
doi = {https://doi.org/10.1016/j.jnnfm.2026.105569},
author = {Santesarti, G. and Marino, M. and Viola, F. and Verzicco, R. and Vairo, G.}
}

@article{HWNW1,
title = {{Development of a predictive mathematical model for coupled {S}tokes/Darcy flows in cross-flow membrane filtration}},
journal = {Chemical Engineering Journal},
volume = {149},
number = {1},
pages = {132-142},
year = {2009},
issn = {1385-8947},
doi = {https://doi.org/10.1016/j.cej.2008.10.012},
author = {Hanspal, N.S. and Waghode, A.N. and Nassehi, V. and Wakeman, R.J.}
}

@article{HWNW2,
title = {{Numerical Analysis of Coupled {S}tokes/Darcy Flows in Industrial Filtrations}},
journal = {Transport in Porous Media},
volume = {64},
pages = {73-101},
year = {2006},
doi = {https://doi.org/10.1007/s11242-005-1457-3},
author = {Hanspal, N.S. and Waghode, A.N. and Nassehi, V. and Wakeman, R.J.}
}

@article{CGHW,
author = {Cao, Y. and Gunzburger, M. and Hua, F. and Wang, X.},
title = {{Coupled {S}tokes-Darcy model with Beavers-Joseph interface boundary condition}},
volume = {8},
journal = {Communications in Mathematical Sciences},
number = {1},
publisher = {International Press of Boston},
pages = {1 - 25},
year = {2010},
}

@article{Beavers_Joseph_1967, title={Boundary conditions at a naturally permeable wall}, volume={30}, 
doi={10.1017/S0022112067001375}, number={1}, 
journal={Journal of Fluid Mechanics}, 
author={Beavers, G.S. and Joseph, D.D.}, 
year={1967}, 
pages={197–207}}

@article{Saffman,
author = {Saffman, P. G.},
title = {{On the Boundary Condition at the Surface of a Porous Medium}},
journal = {Studies in Applied Mathematics},
year = {1971},
volume = {50},
number = {2},
pages = {93-101},
doi = {https://doi.org/10.1002/sapm197150293}
}

@article{lymph,
author = {Antonietti, Paola F. and Bonetti, S. and Botti, M. and Corti, M. and Fumagalli, I. and Mazzieri, I.},
title = {{lymph: Discontinuous Polytopal Methods for Multi-Physics Differential Problems}},
year = {2025},
publisher = {Association for Computing Machinery},
volume = {51},
number = {1},
issn = {0098-3500},
doi = {https://doi.org/10.1145/3716310},
journal = {ACM Transaction on Mathematical Software}
}

@article{AntoniettiBonaldiMazzieri,
    author = {Antonietti, Paola F. and Bonaldi, F. and Mazzieri, I.},
    title = {{A high-order discontinuous {G}alerkin approach to the elasto-acoustic problem}},
    journal = {Computer Methods in Applied Mechanics and Engineering},
    volume = {358},
    pages = {112634},
    year = {2020},
    doi = {https://doi.org/10.1016/j.cma.2019.112634},
}

@article{BonaldiBrennerDroniouMasson,
    author = {Bonaldi, F. and Brenner, K. and Droniou, J. and Masson, R.},
    title = {{Gradient discretization of two-phase flows coupled with mechanical deformation in fractured porous media}},
    journal = {Computers and MAthematics with Applications},
    volume = {98},
    pages = {40-68},
    year = {2021},
    doi = {https://doi.org/10.1016/j.camwa.2021.06.017},
}

@article{KanschatRiviere2010,
    author = {Kanschat, G. and Rivière, B.},
    title = {A strongly conservative finite element method for the coupling of Stokes and Darcy flow},
    journal = {Journal of Computational Physics},
    volume = {229},
    number = {17},
    pages = {5933-5943},
    year = {2010},
    doi = {https://doi.org/10.1016/j.jcp.2010.04.021},
}

@article{CongreveHouston2014,
    author = {Kanschat, G. and Rivière, B.},
    title = {A strongly conservative finite element method for the coupling of Stokes and Darcy flow},
    journal={International Journal of Numerical Analysis and Modeling},
    volume = {11},
    number={3},
    pages={496–524},
    year={2014},
    url={https://www.global-sci.com/ijnam/article/view/10169}, 
}

\appendix

\section{Appendix}\label{app: proof bh cont}
In this section, we provide a complete proof of the continuity of $\bh{\cdot}{\cdot}{}$, i.e.,
\begin{equation}
    |\bh{\vh{v}{}}{\ph{p}{}}{}| \lesssim \tnormdiv{\vh{v}{}}{} \tnormp{\ph{p}{}}{} \qquad \forall \vh{v}{}\in\bbf{X}_h, \forall \ph{p}{}\in M_h.
\end{equation}
\begin{proof} From the definitions of the norms given in Definition \eqref{def: discrete norms} we immediately have
\begin{align}
    \begin{split}
        |\bh{\vh{v}{}}{\ph{p}{}}{\Gamma}| & = \sum_{F\in\Gamma} \int_F |\ph{p}{D} \jumpl\vh{v}{}\jumpr_\bbf{n}| ds \leq \sum_{F\in\Gamma} \|\ph{p}{D}\|_{L^2(F)} \|\jumpl\vh{v}{}\jumpr_\bbf{n}\|_{L^2(F)} \\ & = \sum_{F\in\Gamma} \|\ph{p}{D}\|_{L^2(F)} \|\jumpl\vh{v}{}\jumpr_\bbf{n}\|_{L^2(F)} \left(\frac{h_K}{h_K}\right)^\frac{1}{2} \leq \sum_{K\in\scr{K}:\partial K \cap \Gamma \neq \emptyset}  C_{\gamma} \tnormx{\vh{v}{}}{} \|\ph{p}{D}\|_{L^2(K)} \\ & \lesssim \tnormx{\vh{v}{}}{} \|\ph{p}{D}\|_{L^2(\Omega)},
    \end{split}\\
    \begin{split}
        |\bh{\vh{v}{S}}{\ph{p}{S}}{S}| &= \sum_{K \in \scr{K}_\mcal{S}}\int_K |\ph{p}{S} \nabla\cdot\vh{v}{S}| ds + \sum_{F \in \scr{F}_\mcal{S}}\int_F |\averagel \ph{p}{S} \averager \jumpl\vh{v}{S} \jumpr_\bbf{n}| ds\\ & \leq \|\ph{p}{S}\|_{L^2(\Omega_S)}\|\nabla\cdot\vh{v}{S}\|_{L^2(\Omega_S)} + \sum_{F \in \scr{F}_\mcal{S}} \|\averagel \ph{p}{S} \averager\|_{L^2(F)}\|\jumpl\vh{v}{S}\jumpr_\bbf{n}\|_{L^2(F)}\left(\frac{h_K}{h_K} \right)^\frac{1}{2}\\ & \lesssim \normp{\ph{p}{S}}{S} \normx{\vh{v}{S}}{S} +  \normp{\ph{p}{S}}{S} \normx{\vh{v}{S}}{S} \lesssim \normp{\ph{p}{S}}{S} \normx{\vh{v}{S}}{S},
    \end{split}\\
    \begin{split}
        |\bh{\vh{v}{D}}{\ph{p}{D}}{D}| &= \sum_{K \in \scr{K}_\mcal{D}}\int_K |\ph{p}{D} \nabla\cdot\vh{v}{D}| ds + \sum_{F \in \scr{F}_\mcal{D}}\int_F |\averagel \ph{p}{D} \averager \jumpl\vh{v}{D}\jumpr_\bbf{n}| ds\\ & \leq \|\ph{p}{D}\|_{L^2(\Omega_D)}\|\nabla\cdot\vh{v}{D}\|_{L^2(\Omega_D)} + \sum_{F \in \scr{F}_\mcal{D}} \|\averagel \ph{p}{D} \averager\|_{L^2(F)}\|\jumpl\vh{v}{D}\jumpr_\bbf{n}\|_{L^2(F)}\left(\frac{h_K}{h_K} \right)^\frac{1}{2} \\ & \lesssim \normp{\ph{p}{D}}{D} (\normx{\vh{v}{D}}{D} + \normll{\nabla_h \vh{v}{D}}{D}) + \normp{\ph{p}{D}}{D} \normx{\vh{v}{D}}{D} \\ & \lesssim \normp{\ph{p}{D}}{D} (\normx{\vh{v}{D}}{D} + \normll{\nabla_h \vh{v}{D}}{D}).
    \end{split}
\end{align}
Collecting the above bounds, we obtain
\begin{equation}
    |\bh{\vh{v}{}}{\ph{p}{}}{}|  \lesssim \tnormdiv{\vh{v}{}}{} \tnormp{\ph{p}{}}{}.
\end{equation}
\end{proof}
\section{Appendix}\label{app: proof ah cont}
In this section, we prove the continuity and the monotonicity of $\ah{\cdot}{\cdot}{}$, i.e., 
\begin{equation}
    |a_h(\vh{u}{}, \vh{w}{})-\ah{\vh{v}{}}{\vh{w}{}}{}| \lesssim \tnormx{\vh{u}{}-\vh{v}{}}{}\tnormx{\vh{w}{}}{} \qquad \forall \vh{u}{}, \vh{v}{}, \vh{w}{} \in \bbf{X}_h,
\end{equation}
and
\begin{equation}
    (a_h(\vh{u}{}, \vh{u}{} - \vh{v}{}) - a_h(\vh{v}{}, \vh{u}{} - \vh{v}{})) \gtrsim \tnormxq{\vh{u}{} - \vh{v}{}}{} \qquad \forall \vh{u}{}, \vh{v}{} \in \bbf{X}_h.
\end{equation}
First, we prove the continuity.
\begin{proof}
    We use the Cauchy-Schwarz inequality, the discrete trace inequality \eqref{lemm: trace inverse inequality}, the inverse inequalities, eq. \eqref{eq: g_S bdd and Lip continuous} and eq. \eqref{eq: g_D bdd and Lip continuous}:
\begin{align}
&\left|\sum_{K\in\scr{K}_\mcal{S}}\int_K (g_S(|\bbf{D}(\vh{u}{S})|)\bbf{D}(\vh{u}{S})-g_S(|\bbf{D}(\vh{v}{S})|)\bbf{D}(\vh{v}{S})):\bbf{D}(\vh{w}{S})\ dx \right|
\begin{aligned}[t]
&\leq \overline{\nu}_S \| \bbf{D}(\vh{u}{S}-\vh{v}{S}) \|_{L^2(\Omega_S)} \| \bbf{D}(\vh{w}{S}) \|_{L^2(\Omega_S)} \\
&\lesssim \normx{\vh{u}{S}-\vh{v}{S}}{S}\normx{\vh{w}{S}}{S},
\end{aligned} \\[6pt]
&\Bigg| \sum_{F\in\scr{F}_\mcal{S}} \int_F \averagel (g_S(|\bbf{D}(\vh{u}{S})|)\bbf{D}(\vh{u}{S})\bbf{n}_S-g_S(|\bbf{D}(\vh{v}{S})|)\bbf{D}(\vh{v}{S})\bbf{n}) \averager\cdot 
\begin{aligned}[t]
&\jumpl \vh{w}{S} \jumpr \ ds \Bigg| \lesssim \\ & \lesssim \sum_{K \in \scr{K}_\mcal{S}} \| \nabla (\vh{u}{S}-\vh{v}{S}) \|_{L^2(K)} \| \vh{w}{S} \|_{L^2(K)} \\
&\lesssim \normx{\vh{u}{S}-\vh{v}{S}}{S}\normx{\vh{w}{S}}{S},
\end{aligned} \\[6pt]
&\left| \sum_{F\in\scr{F}_\mcal{S}}\int_F \frac{\sigma_S}{h_K} \jumpl\vh{u}{S}-\vh{v}{S}\jumpr\cdot\jumpl\vh{w}{S}\jumpr\, ds \right|
\begin{aligned}[t]
&\leq \left( \sum_{F\in\scr{F}_\mcal{S}} \frac{\sigma_S}{h_K} \|\jumpl\vh{u}{S}-\vh{v}{S}\jumpr\|_{L^2(F)}^2\right)^{\frac{1}{2}}
\left( \sum_{F\in\scr{F}_\mcal{S}} \frac{\sigma_S}{h_K} \|\jumpl\vh{w}{S}\jumpr\|_{L^2(F)}^2\right)^{\frac{1}{2}} \\
&\leq \normx{\vh{u}{S}-\vh{v}{S}}{S} \normx{\vh{w}{S}}{S},
\end{aligned} \\[6pt]
\end{align}
\begin{align}
& \sum_{j=1}^{d-1}\Bigg| \sum_{F\in\Gamma}\int_F \rho^{-1}((\vh{u}{S}-\vh{v}{S})\cdot\bbf{t}_{\Gamma,j})
\begin{aligned}[t]
&(\vh{w}{S}\cdot\bbf{t}_{\Gamma,j})\, ds \Bigg| \leq
\\ & \leq \sum_{j=1}^{d-1}\left( \sum_{F\in\Gamma} \rho^{-1}\|(\vh{u}{S}-\vh{v}{S})\cdot\bbf{t}_{\Gamma,j}\|^2_{L^2(F)} \right)^{\frac{1}{2}}
\left( \sum_{F\in\Gamma} \rho^{-1}\|\vh{w}{S}\cdot\bbf{t}_{\Gamma,j}\|^2_{L^2(F)} \right)^{\frac{1}{2}},
\end{aligned} \\[6pt]
&\left| \sum_{K\in\scr{K}_\mcal{D}} \int_K (\bbf{K}^{-1}g_D(|\vh{u}{D}|)\vh{u}{D} - \bbf{K}^{-1}g_D(|\vh{v}{D}|)\vh{v}{D})\cdot\vh{w}{D}\, dx \right|
\begin{aligned}[t]
&\leq \frac{\overline{\nu}_D}{k_\text{max}}\|\vh{u}{D}-\vh{v}{D}\| \|\vh{w}{D}\| \\
&\lesssim \normx{\vh{u}{D}-\vh{v}{D}}{D}\normx{\vh{w}{D}}{D},
\end{aligned} \\[6pt]
&\left| \sum_{F\in\scr{F}_\mcal{D}^i}\int_F \frac{\sigma_D}{h_K} \jumpl\vh{u}{D} - \vh{v}{D} \jumpr_\bbf{n} \jumpl\vh{w}{D} \jumpr_\bbf{n}\, ds \right|
\begin{aligned}[t]
&\leq \normx{\vh{u}{D} - \vh{v}{D}}{D}\normx{\vh{w}{D}}{D}, \end{aligned}\\[6pt]
&\left| \sum_{F\in\Gamma} \int_F \frac{\sigma_\Gamma}{h_K}\jumpl\vh{u}{} - \vh{v}{} \jumpr_\bbf{n} \jumpl \vh{w}{} \jumpr_\bbf{n}\, ds \right|
\begin{aligned}[t]
&\leq \tnormx{\vh{u}{} - \vh{v}{}}{}\tnormx{\vh{w}{}}{}, \end{aligned}
\end{align}
where the hidden constants depend on $\overline{\nu_S},\overline{\nu_D}, k_{\max}$, and the inverse inequality constant. Putting together all this estimates we obtain \eqref{eq: Ah bound}.
\end{proof}
Now, we prove the monotonicity.
\begin{proof}
Let $\vh{u}{}, \vh{v}{} \in X_h$, by the discrete Korn's inequality, we can write:
\begin{equation}
    \begin{split}
        \sum_{K \in \scr{K}_\mcal{S}}& \int_K (g_S(|\bbf{D}(\vh{u}{S})|)\bbf{D}(\vh{u}{S}) - g_S(|\bbf{D}(\vh{v}{S})|)\bbf{D}(\vh{v}{S})) : \bbf{D}(\vh{u}{S} - \vh{v}{S}) dx \\ & \gtrsim \|\bbf{D}(\vh{u}{S} - \vh{v}{S}) \|_{L^2(\Omega_S)}^2 \gtrsim \| \nabla_h(\vh{u}{S} - \vh{v}{S})\|_{L^2(\Omega_S)}^2 - \sum_{F \in \scr{F}_\mcal{S}} \frac{1}{h_K} \| \jumpl \vh{u}{S} - \vh{v}{S} \jumpr \|_{L^2(F)}^2.
    \end{split}
\end{equation}
Now we apply the Lipschitz continuity of $g_S(\cdot)$ and Young's inequality, \eqref{eq: g_S bdd and Lip continuous}, to obtain:
\begin{equation}
    \begin{split}
        - \sum_{F \in \scr{F}_\mcal{S}}& \int_F \averagel (g_S(|\bbf{D}(\vh{u}{S})|)\bbf{D}(\vh{u}{S}) - g_S(| \bbf{D}(\vh{v}{S}) |)\bbf{D}(\vh{v}{S}))\bbf{n} \averager \jumpl \vh{u}{S} - \vh{v}{S} \jumpr ds  \\ & \gtrsim - \sum_{F \in \scr{F}_\mcal{S}} \| \averagel \bbf{D}(\vh{u}{S} - \vh{v}{S}) \averager \|_{L^2(F)} \| \jumpl \vh{u}{S} - \vh{v}{S} \jumpr \|_{L^2(F)} \\ & \gtrsim - \frac{1}{4\eta} \| \nabla_h (\vh{u}{S} - \vh{v}{S}) \|_{L^2(\Omega_S)}^2 - \sum_{F \in \scr{F}_\mcal{S}} \frac{\eta}{h_K} \| \jumpl \vh{u}{S} - \vh{v}{S} \jumpr \|_{L^2(F)}^2.
    \end{split}
\end{equation}
Additionally, using the monotonicity of $g_D(\cdot)$, (cf. \eqref{eq: g_D strongly monotone}) we can obtain:
\begin{equation}
    \sum_{K \in \scr{K}_\mcal{D}} \int_K (g_D(|\vh{u}{D}|)\vh{u}{D} - g_D(|\vh{v}{D}|)\vh{v}{D})\cdot (\vh{u}{D} - \vh{v}{D}) dx \gtrsim \| \vh{u}{D} - \vh{v}{D} \|^2_{L^2(\Omega_D)}.
\end{equation}
We can conclude that, for large enough $\sigma_S$, it holds
\begin{equation}
    \begin{split}
        (a_h(\vh{u}{}, \vh{u}{} - \vh{v}{}) - a_h(\vh{v}{}, \vh{u}{} - \vh{v}{}))& \gtrsim \left( 1 - \frac{1}{4\eta} \right) \| \nabla_h(\vh{u}{S} - \vh{v}{S}) \|_{L^2(\Omega_S)}^2 + \sum_{F \in \scr{F}_\mcal{S}} \left(\sigma_S -\frac{1 - \eta}{h_K}\right) \| \jumpl \vh{u}{S} - \vh{v}{S} \jumpr \|_{L^2(F)}^2 \\ &+ \sum_{j=1}^{d-1} \sum_{F \in \Gamma} \rho^{-1} \| (\vh{u}{S} - \vh{v}{S}) \cdot \bbf{t}_{\Gamma, j} \|_{L^2(F)}^2 + \| \vh{u}{D} - \vh{v}{D} \|^2_{L^2(\Omega_D)}\\ & + \sum_{F \in \scr{F}_\mcal{D}^i} \sigma_D \| \jumpl \vh{u}{D} - \vh{v}{D} \jumpr \|_{L^2(F)}^2 +  \sum_{F \in \Gamma} \sigma_\Gamma \| \jumpl (\vh{u}{} - \vh{v}{}) \jumpr_\bbf{n} \|_{L^2(F)}^2 \\ & \gtrsim \tnormxq{\vh{u}{} - \vh{v}{}}{},
    \end{split}
\end{equation}
and the proof is complete.
\end{proof}
\section{Appendix}\label{app: bound div}
In this section, we prove a bound for the $L^2$ norm of the divergence of $\vh{u}{D}$.
We start by considering the second equation of \eqref{eq: discrete SDF} and test it against $\ph{\Tilde{q}}{} = (0,\nabla \cdot \vh{u}{D}$), to obtain
\begin{equation}
    \begin{split}
         - b_h(\vh{u}{}, \ph{\Tilde{q}}{}) + & s_h(\ph{p}{},  \ph{\Tilde{q}}{}) = 
         \sum_{K \in \scr{K}_\mcal{D}} \int_K \nabla \cdot \vh{u}{D} \nabla \cdot \vh{u}{D} dx - \sum_{F \in \scr{F}_\mcal{D}} \int_F \averagel \nabla \cdot \vh{u}{D} \averager \jumpl \vh{u}{D} \jumpr_\bbf{n} ds  \\ &
         - \sum_{F \in \Gamma_h} \int_K \nabla \cdot \vh{u}{D} \jumpl \vh{u}{} \jumpr_\bbf{n} ds
         + \sum_{F \in \scr{F}_\mcal{D}^i} \int_F \xi_D \jumpl \ph{p}{D} \jumpr \jumpl \nabla \cdot \vh{u}{D} \jumpr ds + \sum_{F \in \Gamma} \int_F \xi_\Gamma \jumpl \ph{p}{} \jumpr \nabla \cdot \vh{u}{D} ds= 0.
    \end{split}
\end{equation}
From this we derive:
\begin{equation}
\begin{split}
    \|\nabla \cdot \vh{u}{D}\|^2_{L^2(\Omega_D)} &= \sum_{F \in \scr{F}_\mcal{D}} \int_F \averagel \nabla \cdot \vh{u}{D} \averager \jumpl \vh{u}{D} \jumpr_\bbf{n} ds + \sum_{F \in \Gamma_h} \int_F \nabla \cdot \vh{u}{D} \jumpl \vh{u}{} \jumpr_\bbf{n} ds - \sum_{F \in \scr{F}_\mcal{D}} \int_F \xi_D \jumpl \ph{p}{D} \jumpr \jumpl \nabla \cdot \vh{u}{D} \jumpr ds \\ & - \sum_{F \in \Gamma} \int_F \xi_\Gamma \jumpl \ph{p}{} \jumpr \jumpl \nabla \cdot \vh{u}{D} \jumpr ds \\ & \lesssim \sum_{F \in \scr{F}_\mcal{D}} \| \averagel \nabla \cdot \vh{u}{D} \averager \|_{L^2(F)} \| \jumpl \vh{u}{D} \jumpr_\bbf{n} \|_{L^2(F)}  + \sum_{K \in \Gamma_h} \| \nabla \cdot \vh{u}{D}  \|_{L^2(F)} \| \jumpl \vh{u}{} \jumpr_\bbf{n} \|_{L^2(F)} \\ & + \sum_{F \in \scr{F}_\mcal{D}^i} \xi_D \|\jumpl \ph{p}{D} \jumpr \|_{L^2(F)} \| \jumpl \nabla \cdot \vh{u}{D} \jumpr \|_{L^2(F)} + \sum_{F \in \Gamma} \xi_\Gamma \|\jumpl \ph{p}{} \jumpr \|_{L^2(F)} \| \nabla \cdot \vh{u}{D} \|_{L^2(F)}  \\ & \lesssim \sum_{K \in \scr{K}_\mcal{D}} \frac{1}{\eta} \| \nabla \cdot \vh{u}{D} \|_{L^2(K)}^2 + \sum_{F \in \scr{F}_\mcal{D}} \sigma_D \| \jumpl\vh{u}{D} \jumpr_\bbf{n} \|_{L^2(F)}^2 +  \sum_{K : \partial K \cap \Gamma \neq \emptyset} \frac{1}{\eta} \| \nabla \cdot \vh{u}{D} \|_{L^2(K)}^2 \\ & + \sum_{F \in \Gamma_h} \sigma_\Gamma \| \jumpl\vh{u}{} \jumpr_\bbf{n} \|_{L^2(F)}^2 + \sum_{F \in \scr{F}_\mcal{D}^i} \eta \xi_D \| \jumpl \ph{p}{D} \jumpr \|^2_{L^2(F)} + \sum_{K \in \scr{K}_\mcal{D}} \frac{1}{2} \|\nabla \cdot \vh{u}{D} \|^2_{L^2(K)} \\ & + \sum_{F \in \Gamma} \eta \xi_\Gamma \| \jumpl \ph{p}{} \jumpr \|^2_{L^2(F)} + \sum_{K : \partial K \cap \Gamma \neq \emptyset} \frac{1}{2} \|\nabla \cdot \vh{u}{D} \|^2_{L^2(K)}.
\end{split}
\end{equation}
Finally, we obtain the following estimate.
\begin{equation}
    \begin{split}
        \|\nabla \cdot \vh{u}{D}\|^2_{L^2(\Omega_D)} &\lesssim \sum_{F \in \scr{F}_\mcal{D}} \sigma_D \| \jumpl\vh{u}{D} \jumpr_\bbf{n} \|_{L^2(F)}^2 + \sum_{F \in \Gamma} \sigma_\Gamma \| \jumpl\vh{u}{} \jumpr_\bbf{n} \|_{L^2(F)}^2 + \sum_{F \in \scr{F}_\mcal{D}^i} \xi_D \| \jumpl \ph{p}{D} \jumpr \|^2_{L^2(F)} \\ & + \sum_{F \in \Gamma} \xi_\Gamma \| \jumpl \ph{p}{D} \jumpr \|^2_{L^2(F)}.
    \end{split}
\end{equation}
\end{document}